\documentclass[12pt]{amsart}
\usepackage[margin=1in,letterpaper,portrait]{geometry}
\usepackage{ytableau}
\ytableausetup{notabloids}
\ytableausetup{centertableaux}
\usepackage{latexsym,amsmath,amssymb,verbatim, tikz}
\usepackage{graphicx}
\usepackage[colorlinks]{hyperref}
\usepackage{mathrsfs}
\usepackage{amsfonts}
\usepackage{amsthm,youngtab}
\usepackage{graphicx}
\usepackage{caption}
\usepackage{enumerate,enumitem}
\usepackage{tikz-cd}
\usepackage{rotating}
\usepackage[all]{xy}
\usepackage[titletoc, title]{appendix}
\usepackage{amscd}
\usepackage{float}
\hypersetup{
bookmarksnumbered,
pdfstartview={FitH},
breaklinks=true,
linkcolor=blue,
urlcolor=blue,
citecolor=blue,
bookmarksdepth=2
}	
\usepackage{caption}
\usepackage[utf8x]{inputenc}
\usepackage{dsfont}

\usepackage{stmaryrd} %% Added for todos
\usepackage{pgfplots}
\usepgfplotslibrary{external} 
\usepackage[colorinlistoftodos,bordercolor=orange,backgroundcolor=orange!20,linecolor=orange,textsize=scriptsize]{todonotes}
\makeatletter
\renewcommand{\todo}[2][]{\tikzexternaldisable\@todo[#1]{#2}\tikzexternalenable}
\makeatother

\theoremstyle{plain}
   
   \newtheorem{theorem}{Theorem}[section]

   \newtheorem{lemma}[theorem]{Lemma}
   
   \newtheorem{conjecture}[theorem]{Conjecture}
   
\theoremstyle{definition}
   \newtheorem{definition}[theorem]{Definition}
   
   \newtheorem{example}[theorem]{Example}

   \newtheorem{remark}[theorem]{Remark}

\numberwithin{equation}{section}

\newcommand{\CC}{{\mathbb {C}}}
\newcommand{\QQ}{{\mathbb {Q}}}

\newcommand{\ZZ}{{\mathbb {Z}}}

\newcommand{\ch}{{\operatorname{ch}}}

\newcommand{\SSYT}{{\rm SSYT}}

\DeclareMathOperator*{\wt}{wt}

\newcommand\scalemath[2]{\scalebox{#1}{\mbox{\ensuremath{\displaystyle #2}}}}

\newlength{\mysizetiny}
\newlength{\mysizesmall}
\newlength{\mysize}
\newlength{\mysizelarge}
\begin{document}

\title{Dual canonical bases for unipotent groups and base affine spaces}
\author{Jian-rong Li}
\address{Jian-Rong Li, Faculty of Mathematics, University of Vienna, Oskar-Morgenstern-Platz 1, 1090 Vienna, Austria.} 
\date{}

\maketitle

\begin{abstract}
Denote by $N \subset SL_k$ the subgroup of unipotent upper triangular matrices. In this paper, we show that the dual canonical basis of $\mathbb{C}[N]$ (and base affine spaces) can be parameterized by semi-standard Young tableaux. Moreover, we give an explicit formula for every element in the the dual canonical basis using the data of the corresponding semistandard Young tableau. We apply our results to study cluster variables in $\mathbb{C}[N]$. 
\end{abstract}

\tableofcontents

\section{Introduction}

Quantum groups (or quantized universal enveloping algebras) were introduced independently by Drinfeld \cite{Dri85} and Jimbo \cite{Jim85} around 1985.

Let $\mathfrak{g}$ be a simple complex Lie algebra of type $A,D,E$. Denote by $\mathfrak{g} = \mathfrak{n} \oplus \mathfrak{h} \oplus \mathfrak{n}^-$ a triangular decomposition of $\mathfrak{g}$. Let $q$ be an indeterminate and let $U_q(\mathfrak{g}) = U_q(\mathfrak{n}) \otimes U_q(\mathfrak{h}) \otimes U_q(\mathfrak{n}^-)$ be the Drinfeld-Jimbo quantum group over $\mathbb{C}(q)$. Inspired by a seminal work of Ringel \cite{Ri90}, Lusztig introduced a canonical basis ${\bf B}$ of $U_q(\mathfrak{n})$ with remarkable properties in \cite{Lus90, Lus91}. In \cite{Kas91}, Kashiwara found an alternative approach to the canonical basis of \cite{Lus90} which made sense in the more general context of Kac-Moody Lie algebras. 

The quantum algebra $U_q(\mathfrak{n})$ is endowed with a distinguished scalar product. Let ${\bf B}^*$ be the basis of $U_q(\mathfrak{n})$ adjoint to the canonical basis ${\bf B}$ with respect to this scalar product.
%The basis ${\bf B}^*$ can be identified with a basis of $A_v(\mathfrak{n})$ called the dual canonical basis. 
The dual canonical basis is defined to be image of
the basis ${\bf B}^*$ under the identification of the graded dual of $U_q(\mathfrak{n})$ with $A_q(\mathfrak{n})$. The graded dual $A_q(\mathfrak{n})$ of $U_q(\mathfrak{n})$ can be regarded as the quantum coordinate ring of the unipotent group $N$ with Lie algebra $\mathfrak{n}$ (see e.g. \cite{GLS13, HL15}). 
When $q \to 1$, the basis ${\bf B}^*$ specializes to a basis of the coordinate ring $\CC[N]$ and it is called the dual canonical basis of $\CC[N]$.

Canonical basis and dual canonical basis (in particular, the dual canonical basis of $\CC[N]$) has been studied intensively in the literature using different methods and many important results are obtained, see e.g. \cite{BZ93, BZ96, BFZ96, BFZ05, FZ99, FZ02, FZ03, GHKK18, GLS08a, GLS08, GLS11, GLS13, HL15, Kam07, KKKO18, KL09, KR11, KV00, Nak11, Qin17, Rou08, VV11}.         

On the other hand, more work is needed to give a full description of the dual canonical basis, see e.g. the paragraph before the last paragraph of Section 2 in \cite{GLS08}.

The aim of this paper is to give an explicit description of the dual canonical basis of $\CC[N]$ in the case that $N \subset SL_k$ is the subgroup of unipotent upper triangular matrices, and the dual canonical basis of $\widetilde{\mathbb{C}[SL_k]^{N^-}}$ which is closely related to $\CC[N]$. 

Let $N^- \subset G=SL_k$ be the subgroup of unipotent lower triangular matrices. The group $N^-$ acts on $G$ by left multiplication. Denote by $\mathbb{C}[SL_k]^{N^-}$ the ring of $N^-$-invariant regular functions on $SL_k$. Explicit description of the dual canonical basis of $\mathbb{C}[SL_k]^{N^-}$ is still an open problem, see e.g. the end of Section 6.5 in \cite{FWZ20}. 

The algebra $\mathbb{C}[SL_k]^{N^-}$ is of high importance because it carries exactly one copy of each polynomial $GL_k$ representation exactly once. Thus, this paper is addressing the dual canonical basis for all $GL_k$ representations at once.

Our main result is to give an explicit formula of dual canonical basis elements using the data of semistandard Young tableaux. The dual canonical basis of $\mathbb{C}[N]$ is studied using geometric method in \cite{Ari96, CG97, Gin87, Zel81, Zel85}. The description of the dual canonical basis using semistandard Young tableaux is useful in studying the dual canonical basis combinatorially. For example, it is useful in classifying cluster variables in the dual canonical basis, see Section \ref{sec:application to classification of cluster variables}. We also give a description of mutations in the cluster algebra $\mathbb{C}[N]$ using tableaux. This description agrees with a recent work \cite[Section 7.2]{BDK21} of Bai, Dranowski, and Kamnitzer.

Brundan \cite{Bru06} gave a formula for the entries of the unitriangular transition matrices between the standard monomial and dual canonical bases of the irreducible polynomial representations of $U_q(\mathfrak{gl}_n)$ in terms of Kazhdan–Lusztig polynomials. The main difference between Brundan's result and our result is that we work directly on $\mathbb{C}[N]$ and $\mathbb{C}[SL_k]^{N^-}$. Brundan's methods used quantum Schur-Weyl duality and our method is to apply categorifications of cluster algebras using finite dimensional representations of type A quantum affine algebras.  

The ring $\CC[N]$ has a cluster algebra structure which can be obtained from a cluster algebra structure on $\mathbb{C}[SL_k]^{N^-}$ by identifying leading principal minors with $1$ \cite{FWZ20}. Denote by $\widetilde{\mathbb{C}[SL_k]^{N^-}}$ the quotient of $\mathbb{C}[SL_k]^{N^-}$ by identifying the leading principal minors with $1$. The algebras $\CC[N]$ and $\widetilde{\mathbb{C}[SL_k]^{N^-}}$ have the same cluster algebra structure (cf. Section \ref{subsec:cluster structure on SLkN and CN}).

Denote by ${\rm SSYT}(k-1, [k], \sim)$ a certain quotient of the monoid ${\rm SSYT}(k-1, [k])$ of semi-standard tableaux with at most $k-1$ rows and with entries in $[k]$ (cf.  Section \ref{sec:monoid of semi-standard tableaux}). Our main result is the following. 
\begin{theorem} [{Theorems \ref{thm:chT form dual canonical basis} and \ref{thm: chT for SLkN}}] \label{thm:dual canonical basis are chT}
For a tableau $T$, we define in Section \ref{subs:formula for chT}, a tableau $T'$ each of whose columns is a fundamental tableau with $m$ many columns, a permutation $w_T$ on $m$ letters, and a monomial $\Delta_{u; T'}$ for all permutations $u$ on $m$ letters. 

The set $\{\ch_{\CC[N]}(T): T \in {\rm SSYT}(k-1, [k], \sim)\}$ (respectively, $\{\ch_{\widetilde{\mathbb{C}[SL_k]^{N^-}}}(T): T \in {\rm SSYT}(k-1, [k], \sim)\}$) is the dual canonical basis of $\CC[N]$ (respectively, $\widetilde{\mathbb{C}[SL_k]^{N^-}}$), 
\begin{align*} 
& \ch_{\CC[N]}(T) = \sum_{u \in S_m} (-1)^{\ell(uw_T)} p_{uw_0, w_Tw_0}(1) \Delta_{u; T'} \in \CC[N],  \\
& \ch_{\widetilde{\CC[SL_k]^{N^-}}}(T) = \sum_{u \in S_m} (-1)^{\ell(uw_T)} p_{uw_0, w_Tw_0}(1) \Delta_{u; T'} \in \widetilde{\CC[SL_k]^{N^-}},
\end{align*}  
where $w_0 \in S_m$ is the longest permutation and $p_{y,y'}(t)$ is a Kazhdan-Lusztig polynomial \cite{KL}.
\end{theorem}

The difference between the formulas for $\ch_{\CC[N]}(T)$ and $\ch_{\widetilde{\CC[SL_k]^{N^-}}}(T)$ is that the flag minors in the formula for $\ch_{\CC[N]}(T)$ are flag minors in $\CC[N]$ while the flag minors in the formula for $\ch_{\widetilde{\CC[SL_k]^{N^-}}}(T)$ are flag minors in $\widetilde{\CC[SL_k]^{N^-}}$. We write $\ch_{\CC[N]}(T)$ (respectively, $\ch_{\widetilde{\CC[SL_k]^{N^-}}}(T)$) as $\ch(T)$ if there is no confusion. 

The basic approach of this paper including the proof technique of Theorem \ref{thm:dual canonical basis are chT} is very similar to the approach in \cite{CDFL}. On the other hand, there are differences between the results in this paper and the results in \cite{CDFL}. The formula in Theorem 5.8 in \cite{CDFL} involves only rectangular semistandard tableaux while Theorem \ref{thm:dual canonical basis are chT} involves semistandard tableaux of any shape.

To prove Theorem \ref{thm:dual canonical basis are chT}, we applied Hernandez-Leclerc's monoidal categorification of $\CC[N]$ \cite{HL15}, a $q$-character formula in \cite[Theorem 1.3]{CDFL} which is obtained from a result due to Arakawa-Suzuki \cite{AS} (see also Section 10.1 in \cite{LM18}, and \cite{BaCi, Hen}) and from the quantum affine Schur-Weyl duality \cite{CP96b}, and the following theorem. 

\begin{theorem} [{Theorem \ref{thm: parameterization of simple modules by tableaux}}] 
There is an isomorphism $\mathcal{P}^+_{k, \triangle} \to {\rm SSYT}(k-1, [k],\sim)$ of monoids.
\end{theorem}
Here $\mathcal{P}^+_{k, \triangle}$ is a certain submonoid of the monoid of dominant monomials (cf. Section \ref{subs:monoidal categorification of CN}). 

\begin{remark}
Though our combinatorial results bear a similarity with \cite{CDFL}, they are different: there, the monoid was free on small gap tableau (which happen to correspond to fundamental l-weights) while here they are free on one-column tableaux whose entries are of the form $1, 2, \ldots, p − 1, p+k-i$, where $1 \le p \le i \le k-1$. Although both of small gap tableaux in \cite{CDFL} and fundamental tableaux in this paper correspond to fundamental l-weights, the form of fundamental tableaux in this paper is different than the form of small gap tableaux in \cite{CDFL}. 
\end{remark}

By Theorem \ref{thm:dual canonical basis are chT}, the dual canonical basis of $\CC[N]$ (respectively, $\widetilde{\CC[SL_k]^{N^-}}$) is parametrized by semi-standard tableaux in ${\rm SSYT}(k-1, [k], \sim)$ and every dual canonical basis element is of the form $\ch(T)$ for some $T \in {\rm SSYT}(k-1, [k], \sim)$. In \cite{KKKO18, Qin17}, it is shown that cluster monomials in $\CC[N]$ (respectively, $\widetilde{\CC[SL_k]^{N^-}}$) belong to the dual canonical basis. Therefore every cluster variable in $\CC[N]$ (respectively, $\widetilde{\CC[SL_k]^{N^-}}$) is also of the form $\ch(T)$. 

Denote by $\Delta_J(x) = \Delta_{ \{1, \ldots, |J|\}, J }(x)$ the minor of a matrix $x$ which takes rows $1, \ldots, |J|$ and columns $J$.  
\begin{example}
The cluster variables (not including frozen variables) of $\CC[N]$, $N \subset SL_4$, (respectively, $\widetilde{\CC[SL_4]^N}$) are indexed by the following tableaux:
\begin{align*}
\scalemath{0.9}{
T_1 = \begin{ytableau} 2 \end{ytableau}, \ T_2 = \begin{ytableau} 3 \end{ytableau}, \ T_3 = \begin{ytableau} 1\\ 3 \end{ytableau}, \ T_4 = \begin{ytableau} 1\\ 4 \end{ytableau}, \ T_5 = \begin{ytableau} 2\\ 3 \end{ytableau}, \ T_6 = \begin{ytableau} 2\\ 4 \end{ytableau}, \ T_7 = \begin{ytableau} 1\\ 2\\ 4 \end{ytableau}, \ T_8 = \begin{ytableau} 1\\ 3\\ 4 \end{ytableau}, \ T_9 = \begin{ytableau} 1 & 3\\ 2 \\ 4 \end{ytableau}. }
\end{align*}
In $\widetilde{\CC[SL_4]^N}$ and $\CC[N]$, we have that 
\begin{align*}
& \ch(T_1) = \Delta_2, \ \ch(T_2) = \Delta_3, \ \ch(T_3) = \Delta_{13}, \ \ch(T_4) = \Delta_{14}, \ \ch(T_5) = \Delta_{23}, \\
& \ch(T_6) = \Delta_{24}, \ \ch(T_7) = \Delta_{124}, \ \ch(T_8) = \Delta_{134}, \  \ch(T_9)= \Delta_3 \Delta_{124} - \Delta_4 \Delta_{123}.
\end{align*}

In both of $\CC[N]$ and $\widetilde{\CC[SL_4]^N}$, all flag minors are cluster variables or frozen variables. On the other hand, in both $\CC[N]$ and $\widetilde{\CC[SL_4]^N}$, there is some matrix minor (not flag minor) which is not a cluster variable. In $\CC[N]$, the matrix minor $\Delta_{13,24} = x_{12}x_{34}$ is not a cluster variable. In $\widetilde{\CC[SL_4]^N}$, the matrix minor $\Delta_{13,24} = x_{12} x_{34} - x_{14} x_{32}$ is also not a cluster variable.  

In $\CC[N]$, we have that $\ch(T_9) = \Delta_3 \Delta_{124} - \Delta_4 \Delta_{123} =x_{13}x_{34}-x_{14} = \Delta_{13,34}$. Therefore all cluster variables and frozen variables in $\CC[N]$ are matrix minors. 

In $\widetilde{\CC[SL_4]^N}$, the cluster variable $\ch(T_9) = \Delta_3 \Delta_{124} - \Delta_4 \Delta_{123}$ is not a matrix minor. 
\end{example}

Every tableau $T$ in ${\rm SSYT}(k-1,[k])$ can be written as $T = T'' \cup T'$ where ``$\cup$'' is the multiplication in the monoid ${\rm SSYT}(k-1,[k])$ (cf. Section \ref{sec:monoid of semi-standard tableaux}), $T'$ is a tableau whose columns are fundamental tableaux and $T''$ is a fraction of two trivial tableaux (cf. Section \ref{sec:monoid of semi-standard tableaux}). 

For a tableau $T$ with columns $T_1, \ldots, T_r$, we denote by $\Delta_T = \Delta_{T_1} \cdots \Delta_{T_r}$ the {\sl standard monomial} of $T$. For a fraction $ST^{-1}$ of two tableaux $S,T$, we denote $\Delta_{ST^{-1}} = \Delta_S \Delta_T^{-1}$ (cf. Section \ref{subsec:weights on tableaux and product of flag minors}).

For $T \in {\rm SSYT}(k-1,[k])$, we define $\ch'(T) = \Delta_{T''} \ch_{\widetilde{\CC[SL_k]^{N^-}}}(T')$.
We conjecture that $\{\ch'(T): T \in {\rm SSYT}(k-1,[k])\}$ is the dual canonical basis of $\CC[SL_k]^{N^-}$, see Conjecture \ref{conj:dual canonical basis of CSLkN^-}. 

We also apply our results to classification of cluster variables in $\CC[N]$, $N \subset SL_6$, up to $4$-column tableaux, cf. Section \ref{sec:application to classification of cluster variables}. 

We showed that the numbers of rank $1,2,3,4$ tableaux (not including frozen variables) which are cluster variables in $\CC[N]$, $N \subset SL_6$, are $52, 118, 170, 212$ respectively. Moreover, we found the simplest non-real tableau $T = \begin{ytableau} 1 & 3 \\ 2 & 5 \\ 4 \\ 6 \end{ytableau}$. It corresponds to a prime element $\ch(T)$ in the dual canonical basis of $\CC[N]$ which is not a cluster variable, see (\ref{eq:chT of non-real element in CN}). This tableau corresponds to the simple $U_q(\widehat{\mathfrak{sl}_6})$-module $L(Y_{3, -1} Y_{4, -4} Y_{4, 2} Y_{5, -1} )$, see Theorem \ref{thm: parameterization of simple modules by tableaux}. This module is very similar to the non-real $U_q(\widehat{\mathfrak{sl}_6})$-module $L(Y_{2,-2}Y_{3,-5}Y_{3,1}Y_{4,-2})$ (after translating to the language of dominant monomials, see Section \ref{sec:application to classification of cluster variables}) in Section 2.7 in \cite{Lec}. This module is also similar to the non-real $U_q(\widehat{\mathfrak{sl}_4})$-module $L(Y_{1,4}Y_{2,1}Y_{2,7}Y_{3,4})$ in Section 13.6 in \cite{HL10}. 

We will study the problem of classification of cluster variables in $\CC[N]$ systematically in another work. 

The paper is organized as follows. In Section 2, we give some background on cluster algebras, quantum affine algebras, cluster structure on $\CC[N]$ and $\CC[SL_k]^{N^-}$, and Hernandez-Leclerc's monoidal categorification of $\CC[N]$. In Section 3, we describe the monoid of semi-standard Young tableaux. In Section 4, we show that a certain submonoid of the monoid of dominant monomials is isomorphic to the monoid of semi-standard tableaux. In Section 5, we give a formula for every element in the dual canonical basis of $\CC[N]$ (respectively, $\widetilde{\CC[SL_k]^{N^-}}$). In Section 6, we describe the mutation rule in $\CC[N]$ (respectively, $\widetilde{\CC[SL_k]^{N^-}}$) in terms of tableaux. In Section \ref{sec:application to classification of cluster variables}, we apply our results to classification of cluster variables in $\CC[N]$, $N \subset SL_6$, up to $4$-column tableaux.

\subsection*{Acknowledgements}
The author would like thank the anonymous referee for his/her very helpful comments and suggestions. The author is supported by the Austrian Science Fund (FWF): M 2633-N32 Meitner Program and P 34602 Einzelprojekte.

\section{Preliminary}

\subsection{Cluster algebras}
Fomin and Zelevinsky introduced cluster algebras \cite{FZ02} in order to understand in a concrete and combinatorial way the theory of total positivity (cf. \cite{Lus94, Lus98}) and canonical bases in quantum groups (cf. \cite{Lus90, Lus91, Kas91}). We recall the definition of cluster algebras. 

A quiver $Q$ is an oriented graph given by a set of vertices $Q_0$, a set of arrows $Q_1$, and two maps $s,t: Q_1 \to Q_0$ taking an arrow to its source and target, respectively.

Let $Q$ be a finite quiver without loops or $2$-cycles. For a vertex $k \in Q_0$, the {\sl mutated quiver} $\mu_k(Q)$ is a quiver with the same set of vertices as $Q$, and its set of arrows is obtained by
the following procedure:
\begin{enumerate}
\item[(i)] add a new arrow $i \to j$ for every existing pair of arrows $i \to k$, $k \to j$; 

\item[(ii)] reverse the orientation of every arrow with target or source equal to $k$,

\item[(iii)] erase every pair of opposite arrows possibly created by (i). 
\end{enumerate}

Let $m \ge n$ be positive integers and let $\mathcal{F}$ be an ambient field of rational functions in $n$ independent variables over $\QQ(x_{n+1}, \ldots, x_m)$. A {\sl seed} in $\mathcal{F}$ is a pair $({\bf x}, Q)$, where ${\bf x} = (x_1, \ldots, x_m)$ is a free generating set of $\mathcal{F}$, and $Q$ is a quiver (without loops or $2$-cycles) with vertices $[m]$ whose vertices $1, \ldots, n$ are called mutable and whose vertices $n+1, \ldots, m$ are called frozen. For a seed $({\bf x}, Q)$ in $\mathcal{F}$ and $k \in [n]$, the {\sl mutated seed} $\mu_k({\bf x}, Q)$ in direction $k$ is $({\bf x}', \mu_k(Q))$, where ${\bf x}' = (x_1', \ldots, x_m')$ with $x_j'=x_j$ for $j\ne k$ and $x_k' \in \mathcal{F}$ is determined by the {\sl exchange relation}:
\begin{align*}
x_k' x_k = \prod_{\alpha \in Q_1, s(\alpha)=k} x_{t(\alpha)} + \prod_{\alpha \in Q_1, t(\alpha)=k} x_{s(\alpha)}.
\end{align*} 

The mutation class of a seed $({\bf x}, Q)$ is the set of all seeds obtained from $({\bf x}, Q)$ by a finite sequence of mutations. For every seed $((x_1', \ldots, x_n', x_{n+1}, \ldots, x_m), Q')$ in the mutation class, the set $\{x_1', \ldots, x_n', x_{n+1}, \ldots, x_m\}$ is called a cluster, $x_1', \ldots, x_n'$ are called cluster variables, and $x_{n+1}, \ldots, x_m$ are called frozen variables. The cluster algebra $\mathcal{A}({\bf x}, Q)$ is the $\ZZ[x_{n+1}, \ldots, x_{m}]$-subalgebra of $\mathcal{F}$ generated by all cluster variables. A {\sl cluster monomial} is a product of non-negative powers of cluster variables belonging to the same cluster.

\subsection{\texorpdfstring{\text{Cluster structure on $\CC[N]$ and $\CC[SL_k]^{N^-}$}}{Cluster structure on the coordinate ring of base affine space and unipotent group}}\label{subsec:cluster structure on SLkN and CN}

In this subsection, we recall the cluster structure on $\CC[N]$ and $\CC[SL_k]^{N^-}$, cf. \cite{BFZ96, BFZ05, FZ99, FZ03, GLS08}. 

Let $V \cong \CC^k$ be a $k$-dimensional complex vector space. By choosing a basis in $V$, one can identify $G = SL_k$ with the special linear group $SL(V)$ complex matrices with determinant $1$. The subgroup $N^- \subset G$ of unipotent lower triangular matrices acts on $G$ by left multiplication. This action induces an action of $N^-$ on the coordinate ring $\CC[G]$. Denote by $\CC[G]^{N^-}$ the ring of $N^-$-invariant regular functions on $G$. The ring $\CC[SL_k]^{N^-}$ has a cluster algebra structure whose initial cluster is given as follows.

For a $n \times n$ matrix $z$ and $J',J \subset [n]$ ($|J'| = |J|$), denote by $\Delta_{J', J}(z)$ the determinant of the submatrix of $z$ with rows labeled by $J'$ and columns labeled by $J$. In the case that $J' = \{1,2,\ldots, |J|\}$, we write $\Delta_{J} = \Delta_{J', J}$ and it is called a {\sl flag minor}. 

Let $I=[k-1]$ be the set of the vertices of the Dynkin diagram of $\mathfrak{sl}_k$. Let $Q_{k, \triangle}$ be a quiver with the vertex set $V_{k, \triangle}=\{ (i, p): i \in I \cup \{k\}, p \in [i] \} \setminus \{(k,k)\}$ and with edge set: 
\begin{align*}
& (i,p) \to (i+1,p+1), \quad (i, p) \to (i, p-1), \quad (i, p) \to (i-1, p),
\end{align*}
see Figure \ref{fig:initial cluster C[N] k is 5}. The vertices $(i,i)$, $i \in I$ and $(k, p)$, $p \in I$ are frozen. 

For $i \in I$, $p \in [i]$, denote $\Delta^{(i,p)} = \Delta_J$, where $J = \{1,2,\ldots, p-1, p+k-i\}$. Attach to the vertex $(i,p)$ the flag minor $\Delta^{(i,p)}$, $i \in I$, $p \in [i]$. An initial cluster of $\CC[SL_k]^{N^-}$ consists of the initial quiver $Q_{k, \triangle}$ and initial cluster variables $\Delta^{(i,p)}$, $i \in I$, $p \in [i]$. Figure \ref{fig:initial cluster C[N] k is 5} is the initial cluster for $\CC[SL_k]^{N^-}$ ($k=5$) if we replace $\Delta_{1,\ldots,i}=1$ by $\Delta_{1,\ldots,i}$, $i \in [k-1]$. 

In Figure \ref{fig:initial cluster C[N] k is 5}, 
\[
\Delta_5, \Delta_4, \Delta_3, \Delta_2, \Delta_1, \Delta_{45}, \Delta_{34}, \Delta_{23}, \Delta_{12}, \Delta_{345}, \Delta_{234}, \Delta_{123}, \Delta_{2345}, \Delta_{1234},
\]
sit at the vertices 
\[
(1,1), (2,1), (3,1), (4,1), (5,1), (2,2), (3,2), (4,2), (5,2), (3,3), (4,3), (5,3), (4,4), (5,4), 
\]
respectively.

\begin{figure}
\begin{tikzpicture}[scale=1.5]
		\node  (1) at (-1, 1) {\fbox{$\Delta_{1}=1$}};
		\node  (2) at (-2, -0) {$\Delta_{2}$};
		\node  (3) at (0, -0) {\fbox{$\Delta_{12}=1$}};
		\node  (4) at (-3, -1) {$\Delta_{3}$};
		\node  (5) at (-1, -1) {$\Delta_{23}$};
		\node  (6) at (1, -1) {\fbox{$\Delta_{123}=1$}};
		\node  (7) at (-4, -2) {$\Delta_{4}$};
		\node  (8) at (-2, -2) {$\Delta_{34}$};
		\node  (9) at (0, -2) {$\Delta_{234}$};
		\node  (10) at (2, -2) {\fbox{$\Delta_{1234}=1$}};
		\node  (11) at (-5, -3) {\fbox{$\Delta_{5}$}};
		\node  (12) at (-3, -3) {\fbox{$\Delta_{45}$}};
		\node  (13) at (-1, -3) {\fbox{$\Delta_{345}$}};
		\node  (14) at (1, -3) {\fbox{$\Delta_{2345}$}};
%		\node  (15) at (3, -3) {\fbox{$\Delta_{12345}=1$}};
     
     \draw[->] (1)--(2);
     \draw[->] (2)--(3);
%     \draw[->] (3)--(1);
     \draw[->] (2)--(4);
     \draw[->] (4)--(5);
     \draw[->] (5)--(6);
     \draw[->] (5)--(2);
%     \draw[->] (6)--(3);
     \draw[->] (3)--(5);
     \draw[->] (4)--(7);
     \draw[->] (7)--(8);
     \draw[->] (8)--(9);
     \draw[->] (9)--(10);
     \draw[->] (8)--(4);
     \draw[->] (5)--(8);
     \draw[->] (6)--(9);
     \draw[->] (9)--(5);
%     \draw[->] (10)--(6);
     \draw[->] (7)--(11);
%     \draw[->] (11)--(12);
%     \draw[->] (12)--(13);
%     \draw[->] (13)--(14);
%     \draw[->] (14)--(15);
     \draw[->] (12)--(7);
     \draw[->] (13)--(8);
     \draw[->] (14)--(9);
%     \draw[->] (15)--(10);
     \draw[->] (8)--(12);
     \draw[->] (9)--(13);
%     \draw[->] (10)--(14); 
\end{tikzpicture}
            \caption{The initial cluster for $\CC[N]$ ($N \subset SL_5$) and $\widetilde{\mathbb{C}[SL_5]^{N^-}}$. This is also the initial cluster for $\CC[SL_5]^{N-}$ if we replace $\Delta_{1,\ldots,i}=1$ by $\Delta_{1,\ldots,i}$, $i \in [4]$.}
            \label{fig:initial cluster C[N] k is 5}
\end{figure}

Denote by $\widetilde{\mathbb{C}[SL_k]^{N^-}}$ the quotient of $\mathbb{C}[SL_k]^{N^-}$ by identifying the leading principal minors $\Delta_{1,\ldots, i}$ ($i \in [k-1]$) with $1$. The cluster algebra structure on $\mathbb{C}[SL_k]^{N^-}$ induces a cluster algebra structure on $\widetilde{\mathbb{C}[SL_k]^{N^-}}$.

Denote by $N \subset SL_k$ the subgroup of unipotent upper triangular matrices. The ring map $\mathbb{C}[SL_k]^{N^-} \to \CC[N]$
defined by restricting $N^-$-invariant functions on $SL_k$ to the subgroup $N$. This map is onto and transforms the above described cluster structure on $\mathbb{C}[SL_k]^{N^-}$ into a cluster structure on $\CC[N]$ (cf. \cite{FWZ20}). This cluster structure on $\CC[N]$ has an initial cluster consisting of the initial quiver $Q_{k, \triangle}$ and initial cluster variables $\Delta^{(i,p)}$, $i \in I$, $p \in [i]$, see Figure \ref{fig:initial cluster C[N] k is 5}.

\subsection{\texorpdfstring{Monoidal categorification of the cluster algebra structure on $\CC[N]$}{Monoidal categorification of the cluster algebra structure for unipotent group}} \label{subs:monoidal categorification of CN}

%In this section, we recall the standard facts about finite dimensional $U_q(\widehat{\mathfrak{sl}_k})$-modules, $q$-characters, and Hernandez-Leclerc's category $\mathcal{C}_{k, \triangle}$, see \cite{CP95a, FR, HL15}.

%Let $\mathcal{C}$ be the category of finite dimensional $U_q(\widehat{\mathfrak{g}})$-modules. 

Hernandez and Leclerc introduced the notion of a monoidal categorification
of a cluster algebra in \cite{HL10, Kas18}. For a monoidal category $(\mathcal{C}, \otimes)$, a simple object $S$ of $C$ is called {\sl real} if $S \otimes S$ is simple. A simple object $S$ is called {\sl prime} if there exists no non-trivial factorization $S \cong S_1 \otimes S_2$. The monoidal category $\mathcal{C}$ is called a {\sl monoidal categorification} of a cluster algebra $\mathcal{A}$ if the Grothendieck ring of $\mathcal{C}$ is isomorphic to $\mathcal{A}$ and if (1) any cluster monomial of A corresponds to the class of a real simple object of $\mathcal{C}$, and (2)
any cluster variable of $\mathcal{A}$ corresponds to the class of a real simple prime object of $\mathcal{C}$. 

Let $Q$ be an orientation of the Dynkin diagram of $\mathfrak{g}$. Hernandez and Leclerc \cite{HL15} constructed a tensor category $\mathcal{C}_Q$ and showed that $\mathcal{C}_Q$ is a monoidal categorification of the ring $\mathbb{C}[N]$ and its dual canonical basis. To our purpose, we use a special case $\mathcal{C}_{k, \triangle}$ of $\mathcal{C}_Q$. We recall the definition of $\mathcal{C}_{k, \triangle}$ in the following.

%In \cite{HL10}, \cite{HL15}, Hernandez and Leclerc introduced a certain full subcategory $\mathcal{C}_{k, \triangle}$ of $\mathcal{C}$. We recall the definition for $\mathfrak{g}=\mathfrak{sl}_k$. 

Let $\mathfrak{g}$ be a simple Lie algebra and $I$ the set of the vertices of the Dynkin diagram of $\mathfrak{g}$. Denote by $P$ the {\sl weight lattice} of $\mathfrak{g}$ and by $Q \subset P$ the {\sl root lattice} of $\mathfrak{g}$. There is a partial order on $P$ given by $\lambda \le \lambda'$ if and only if $\lambda' - \lambda$ is equal to a non-negative integer linear combination of positive roots.

In this paper, we take $q$ to be a non-zero complex number which is not a root of unity, $\mathfrak{g}=\mathfrak{sl}_k$, and $I=[k-1]$ be the set of vertices of the Dynkin diagram of $\mathfrak{g}$. The {\sl quantum affine algebra} $U_q(\widehat{\mathfrak{g}})$ is a Hopf algebra that is a $q$-deformation of the universal enveloping algebra of $\widehat{\mathfrak{g}}$ \cite{Dri85, Dri87, Jim85}.

We fix $a \in \CC^{\times}$ and denote $Y_{i,s} = Y_{i,aq^s}$, $i \in I$, $s \in \ZZ$. Denote by $\mathcal{P}$ the free abelian group generated by $Y_{i,s}^{\pm 1}$, $i \in I$, $s \in \ZZ$, denote by $\mathcal{P}^+$ the submonoid of $\mathcal{P}$ generated by $Y_{i,s}$, $i \in I$, $s \in \ZZ$, and denote by $\mathcal{P}^+_{k,\triangle}$ the submonoid of $\mathcal{P}^+$ generated by $Y_{i,i-2p}$, $i \in I$, $p \in [i]$. An object $V$ in $\mathcal{C}_{k,\triangle}$ is a finite dimensional $U_q(\widehat{\mathfrak{sl}_k})$-module which satisfies the condition: for every composition factor $S$ of $V$, the highest $l$-weight of $S$ is a monomial in $Y_{i, i-2p}$, $i\in I$, $p \in [i]$. Simple modules in $\mathcal{C}_{k, \triangle}$ are of the form $L(M)$ (cf. \cite{CP95a}, \cite{HL10}), where $M \in \mathcal{P}^+_{k,\triangle}$ and $M$ is called the highest \textit{$l$-weight} of $L(M)$. The elements in $\mathcal{P}^+$ are called {\sl dominant monomials}. Denote by $K(\mathcal{C}_{k,\triangle})$ the Grothendieck ring of $\mathcal{C}_{k,\triangle}$. 

Let $\mathbb{Z}\mathcal{P} = \mathbb{Z}[Y_{i, s}^{\pm 1}]_{i\in I, s\in \mathbb{Z}}$ be the group ring of $\mathcal{P}$. The $q$-{\sl character} of a $U_q(\widehat{\mathfrak{g}})$-module $V$ is given by (cf. \cite{FR})
\begin{align*}
\chi_q(V) = \sum_{m\in \mathcal{P}} \dim(V_{m}) m \in \mathbb{Z}\mathcal{P},
\end{align*}
where $V_m$ is the $l$-weight space with $l$-weight $m$ ($l$-weights of $V$ are identified with monomials in $\mathcal{P}$). It is shown in \cite{FR} that $q$-characters characterize simple $U_q(\widehat{\mathfrak{g}})$-modules up to isomorphism.

Denote by $\wt: \mathcal{P} \to P$ the group homomorphism defined by sending $Y_{i,a}^{\pm} \mapsto \pm \omega_i$, $i \in I$, where $\omega_i$'s are fundamental weights of $\mathfrak{g}$. For a finite dimensional simple $U_q(\widehat{\mathfrak{g}})$-module $L(M)$, we write $\wt(L(M)) = \wt(M)$ and call it the highest weight of $L(M)$. 

Let $\mathcal{Q}^+$ be the monoid generated (in the case that $\mathfrak{g} = \mathfrak{sl}_k$) by
\begin{align} \label{eq:Aia}
A_{i,s} = Y_{i,s+1}Y_{i,s-1} \prod_{j \in I, |j-i|=1} Y_{j,s}^{-1}, \quad i\in I, \ s\in \mathbb{Z}.
\end{align}
There is a partial order $\leq$ on $\mathcal{P}$ (cf. \cite{FM,Nak00}) defined by
\begin{align}
M \leq M' \text{ if and only if } M'M^{-1}\in \mathcal{Q}^{+}. \label{partial order of monomials}
\end{align}

For $i \in I$, $s \in \mathbb{Z}$, $k \in \ZZ_{\ge 1}$, the modules $L(X_{i,k}^{(s)})$, where $X_{i,k}^{(s)} = Y_{i,s} Y_{i,s+2} \cdots Y_{i,s+2k-2}$, are called {\sl Kirillov-Reshetikhin modules}. The modules $L(X_{i,1}^{(s)}) = L(Y_{i,s})$ are called {\sl fundamental modules}.

Hernandez and Leclerc \cite{HL15} proved that the tensor category $\mathcal{C}_{k,\triangle}$ is a monoidal categorification of the ring $\CC[N]$ and its dual canonical basis. The Grothendieck ring $K(\mathcal{C}_{k,\triangle})$ has a cluster algebra structure with an initial seed consisting of the initial quiver $Q_{k,\triangle}$ and initial cluster variables $X_{i, p}^{(i-2p)}$, $i \in I$, $p \in [i]$, where $X_{i, p}^{(i-2p)}$ sits at the position $(i,p)$ of the quiver $Q_{k,\triangle}$, see Figure \ref{fig:initial cluster labeled by modules}. We put trivial modules $\CC$ at the positions $(k, i)$, $i \in [k-1]$, in order to compare with the quiver in Figure \ref{fig:initial cluster C[N] k is 5}. 
  
Recall that in Section \ref{subsec:cluster structure on SLkN and CN}, for $i \in I$, $p \in [i]$, we denote $\Delta^{(i,p)} = \Delta_J$, where $J = \{1,2,\ldots, p-1, p+k-i\}$.

\begin{theorem} [{\cite[Theorems 1.1, 1.2, and 6.1]{HL15}}] \label{thm:isomorphism between Grothendieck ring and CN}
The assignments $L(Y_{i,i-2p}) \mapsto \Delta^{(i,p)}$, $i \in I$, $p \in [i]$, induce an algebraic isomorphism $\Phi_{\CC[N]}: K(\mathcal{C}_{k,\triangle}) \to \mathbb{C}[N]$. 

The assignments $L(Y_{i,i-2p}) \mapsto \Delta^{(i,p)}$, $i \in I$, $p \in [i]$, induce an algebraic isomorphism $\Phi_{\widetilde{\mathbb{C}[SL_k]^{N^-}}}: K(\mathcal{C}_{k,\triangle}) \to \widetilde{\mathbb{C}[SL_k]^{N^-}}$. 
\end{theorem} 

We usually write $\Phi_{\CC[N]}$ (respectively, $\Phi_{\widetilde{\mathbb{C}[SL_k]^{N^-}}}$) as $\Phi$ if there is no confusion.

\begin{figure}
\begin{tikzpicture}[scale=1.8]
		\node  (1) at (-1, 1) {\fbox{$\CC$}};
		\node  (2) at (-2, -0) {$Y_{4,2}$};
		\node  (3) at (0, -0) {\fbox{$\CC$}};
		\node  (4) at (-3, -1) {$Y_{3,1}$};
		\node  (5) at (-1, -1) {$Y_{4,2}Y_{4,0}$};
		\node  (6) at (1, -1) {\fbox{$\CC$}};
		\node  (7) at (-4, -2) {$Y_{2,0}$};
		\node  (8) at (-2, -2) {$Y_{3,-1}Y_{3,1}$};
		\node  (9) at (0, -2) {$Y_{4,-2}Y_{4,0}Y_{4,2}$};
		\node  (10) at (2, -2) {\fbox{$\CC$}};
		\node  (11) at (-5, -3) {\fbox{$Y_{1,-1}$}};
		\node  (12) at (-3, -3) {\fbox{$Y_{2,-2}Y_{2,0}$}};
		\node  (13) at (-1, -3) {\fbox{$Y_{3,-3}Y_{3,-1}Y_{3,1}$}};
		\node  (14) at (1, -3) {\fbox{$Y_{4,-4}Y_{4,-2}Y_{4,0}Y_{4,2}$}};
%		\node  (15) at (3, -3) {\fbox{$\CC$}};
     
     \draw[->] (1)--(2);
     \draw[->] (2)--(3);
%     \draw[->] (3)--(1);
     \draw[->] (2)--(4);
     \draw[->] (4)--(5);
     \draw[->] (5)--(6);
     \draw[->] (5)--(2);
%     \draw[->] (6)--(3);
     \draw[->] (3)--(5);
     \draw[->] (4)--(7);
     \draw[->] (7)--(8);
     \draw[->] (8)--(9);
     \draw[->] (9)--(10);
     \draw[->] (8)--(4);
     \draw[->] (5)--(8);
     \draw[->] (6)--(9);
     \draw[->] (9)--(5);
%     \draw[->] (10)--(6);
     \draw[->] (7)--(11);
%     \draw[->] (11)--(12);
%     \draw[->] (12)--(13);
%     \draw[->] (13)--(14);
%     \draw[->] (14)--(15);
     \draw[->] (12)--(7);
     \draw[->] (13)--(8);
     \draw[->] (14)--(9);
%     \draw[->] (15)--(10);
     \draw[->] (8)--(12);
     \draw[->] (9)--(13);
%     \draw[->] (10)--(14); 
\end{tikzpicture}
            \caption{The initial cluster for $\mathcal{C}_{5, \triangle}$.}
            \label{fig:initial cluster labeled by modules}
\end{figure}

\section{The monoid of semi-standard Young tableaux} \label{sec:monoid of semi-standard tableaux}
In this section, we show that the set of semi-standard Young tableaux with at most $k$ rows and with entries in a set $[m]$ form a monoid under certain product ``$\cup$''. 

For $k, m \in \ZZ_{\ge 1}$, denote by ${\rm SSYT}(k, [m])$ the set of all semi-standard Young tableaux (including the empty tableau denoted by $\mathds{1}$) with less or equal to $k$ rows and with entries in $[m]$. For a tableau $T \in {\rm SSYT}(k, [m])$ with $k'$ ($k' \le k$) rows, when we say the $i$th $(i>k')$ row of $T$, we understand that the $i$th row is empty. 

%We call a column of a tableau in ${\rm SSYT}(k, [m])$ non-full if the number of rows of the column is less than $k$. 

%For a semi-standard tableau $T$ in ${\rm SSYT}(k, [m])$, we define a tableau $T^{(a)}$ which is obtained from $T$ by adding numbers $s, \ldots, s+k$ to $T$ in the first non-full column, and then the second non-full column, and so on until all columns are full, where $s-1$ is the maximal number in $T$ and $k$ is some integer. We define a function $f_a: {\rm SSYT}(k, [m]) \to {\rm SSYT}(k, [\infty])$ by $T \mapsto T^{(a)}$. 

%The following is immediate. 
%\begin{lemma} \label{lem:tableau T is in SSYTnm if and only if fa of T is in SSYTn infinity}
%For any tableau with entries in $[m]$, $T \in {\rm SSYT}(k, [m])$ if and only if  $T^{(a)} \in {\rm SSYT}(k, [\infty])$.
%\end{lemma}

For $T,T' \in {\rm SSYT}(k, [m])$, we denote by $T \cup T'$ the row-increasing tableau whose $i$th row is the union of the $i$th rows of $T$ and $T'$ (as multisets). 

\begin{example}
In ${\rm SSYT}(5, [6])$, we have that 
\begin{align*}
\begin{ytableau} 1 \\ 4 \\ 5 \end{ytableau} \cup \begin{ytableau} 2 \\ 3 \end{ytableau} = \begin{ytableau} 1 & 2 \\ 3 & 4 \\ 5 \end{ytableau}.
\end{align*} 
\end{example}

For $S, T \in {\rm SSYT}(k, [m])$, we say that $S$ is a {\sl factor} of $T$ (denoted by $S \subset T$) if for every $i \in [k]$, the $i$th row of $S$ is contained in the $i$th row of $T$ (as multisets). For a factor $S$ of $T$, we define $\frac{T}{S}=S^{-1}T=TS^{-1}$ to be the row-increasing tableau whose elements in the $i$th row are the elements in the multiset-difference of $i$th row of $T$ and the $i$th row of $S$, for every $i \in [k]$.

We call a tableau $T \in {\rm SSYT}(k, [m])$ {\sl trivial} if it is a one-column tableau with entries $\{1,\ldots, p\}$ for some $p \in [k]$. For any $T \in {\rm SSYT}(k, [m])$, we denote by $T_{\text{red}} \subset T$ the semi-standard tableau obtained by removing a maximal trivial factor from $T$. For $S, T \in {\rm SSYT}(k, [m])$, define $S \sim T$ if $S_{\text{red}} = T_{\text{red}}$. Note that if $T \sim T'$, then $T, T'$ have the same number of rows. It is clear that~``$\sim$'' is an equivalence relation. We denote by ${\rm SSYT}(k, [m],\sim)$ the set of $\sim$-equivalence classes in ${\rm SSYT}(k, [m])$. With a slight abuse of notation, we write $T \in {\rm SSYT}(k, [m],\sim)$ instead of $[T] \in {\rm SSYT}(k, [m],\sim)$.

In \cite[Lemma 3.6]{CDFL}, we proved that the set of all semi-standard Young tableaux of rectangular shape with $k$ rows and with entries in $[m]$ is a monoid with the multiplication ``$\cup$''. Similarly, we have the following result.

\begin{lemma} \label{lem:SSYT is a monoid}
The set ${\rm SSYT}(k, [m])$ (respectively, ${\rm SSYT}(k, [m], \sim)$) form a commutative cancellative monoid with the multiplication $``\cup''$.
\end{lemma}

\begin{proof}
It is clear that the set ${\rm SSYT}(k, [m])$ form a commutative cancellative monoid implies that the set ${\rm SSYT}(k, [m], \sim)$ form a commutative cancellative monoid. Therefore it suffices to prove the result for ${\rm SSYT}(k, [m])$.

By definition, $``\cup''$ is commutative and associative. Suppose that $A,T,T' \in {\rm SSYT}(k, [m])$ and $A \cup T = A \cup T'$. For every $i \in [k]$, the $i$th row of $T$ (respectively, $T'$) is obtained from the $i$th row of $A \cup T$ (respectively, $A \cup T'$) by removing elements in the $i$th row of $A$ (as multisets). Since $A \cup T = A \cup T'$, we have that the $i$th rows of $T, T'$ are the same for every $i \in [k]$. Therefore $T = T'$. 

We now prove that for $T, T' \in {\rm SSYT}(k, [m])$, we have $T \cup T' \in {\rm SSYT}(k, [m])$. Denote by $S(i)$ the $i$th row of a tableau $S$. We need to prove that for any $i < j$, the $2$-row tableau with the first row $T(i) \cup T'(i)$ and the second row $T(j) \cup T'(j)$ is semi-standard. It suffices to prove this in the case that $T'$ has one column. Let $i, j$ rows of $T$ be
\begin{align*}
\begin{matrix}
a_1 & a_2 & \cdots & a_{r_1} \\
b_1 & b_2 & \cdots & b_{r_2}, 
\end{matrix}
\end{align*}
for some $r_1 \ge r_2$. We have the following cases. 

{\bf Case 1.} $T'$ does not have entry in rows $i$ and $j$. In this case, the result is trivial.

{\bf Case 2.} $T'$ has an entry $a'$ in row $i$ and the row $j$ is empty. There exists $k \in [0, r_1]$ such that $a_1 \le \cdots \le a_k \le a' \le a_{k+1} \le \cdots \le a_{r_1}$. The $i,j$ rows of $T \cup T'$ are 
\begin{equation*}
\begin{matrix}
a_1 & a_2 & \cdots & a_k & a' & a_{k+1} & \cdots & a_{r_1} \\
b_1 & b_2 & \cdots & b_{k} & b_{k+1} & b_{k+2} & \cdots & b_{r_2}.
\end{matrix}
\end{equation*}
We have that $a' \le a_{k+1}<b_{k+1}$ and for all $d \in [k+1, r_2-1]$, $a_d < b_d \le b_{d+1}$. Therefore the $i,j$ rows of $T \cup T'$ form a $2$-row semi-standard tableau. 

{\bf Case 3.} $T'$ has entries $a'$ and $b'$ in rows $i$ and $j$. 
There are $k \in [0, r_1]$, $l \in [0, r_2]$ such that $a_1 \le \cdots \le a_k \le a' \le a_{k+1} \le \cdots \le a_{r_1}$ and $b_1 \le \cdots \le b_l \le b' \le b_{k+1} \le \cdots \le b_{r_2}$. 

If $k=l$, then the $i,j$ rows of $T \cup T'$ form a $2$-row semi-standard tableau. If $k > l$, then the $i,j$ rows of $T \cup T'$ are 
\begin{align*}
\begin{array}{cccccccccccc}
a_1 & a_2 & \cdots & a_l & a_{l+1} & a_{l+2} & \cdots & a_k & a' & a_{k+1} & \cdots & a_{r_1} \\
b_1 & b_2 & \cdots & b_l & b' & b_{l+1} & \cdots & b_{k-1} & b_{k} & b_{k+1} & \cdots & b_{r_2}.
\end{array}
\end{align*}
We have $a' < b' \le b_k$, $a_{l+1} \le a' < b'$, and for all $d \in [l+2, k]$, $a_d \le a' < b' \le b_{d-1}$. Therefore the $i,j$ rows of $T \cup T'$ form a $2$-row semi-standard tableau. 

If $k < l$, then the $i,j$ rows of $T \cup T'$ are 
\begin{align*}
\begin{array}{cccccccccccc}
a_1 & a_2 & \cdots & a_k & a' & a_{k+1} & \cdots & a_{l-1} & a_l & a_{l+1} & \cdots & a_{r_1} \\
b_1 & b_2 & \cdots & b_k & b_{k+1} & b_{k+2} & \cdots & b_{l} & b' & b_{l+1} & \cdots & b_{r_2}.
\end{array}
\end{align*}
We have $a' \le a_{k+1} < b_{k+1}$, $a_{l} < b_l \le b'$, and for all $d \in [k+1, l-1]$, $a_d < b_d \le b_{d+1}$. Therefore the $i,j$ rows of $T \cup T'$ form a $2$-row semi-standard tableau. 
\end{proof}

\section{Isomorphisms of monoids $\mathcal{P}_{k,\triangle}^{+}$ and ${\rm SSYT}(k-1, [k], \sim)$}
In this section, we show that the monoids $\mathcal{P}_{k,\triangle}^{+}$ and ${\rm SSYT}(k-1, [k], \sim)$ are isomorphic. 

\subsection{Factorization of a tableau as a product of fundamental tableaux}

For $i \in I$, $p \in [i]$, denote by $T^{(i,p)}$ the one-column tableau with entries $\{1,2,\ldots, p-1, p+k-i\}$. We call the tableau $T^{(i,p)}$ a {\sl fundamental tableau}. We also use $T_{(l, a)}$ to denote a fundamental tableau with $l$ rows and whose last entry $a$. We have that $T_{(l, a)} = T^{(l+k-a, l)}$. 

There is a total order on the set of one-column fundamental tableaux in ${\rm SSYT}(k, [m])$: for two one column fundamental tableaux $T=T_{(l, a)}, T'=T_{(l', a')}$, $T \le T'$ if either $l>l'$ or $l=l'$, $a \le a'$. For example, 
\begin{align*}
\begin{ytableau}
1 \\ 2 \\ 5
\end{ytableau} < \begin{ytableau}
1 \\ 2 \\ 6
\end{ytableau} < \begin{ytableau}
1 \\ 3
\end{ytableau} < \begin{ytableau}
1 \\ 4
\end{ytableau} < \begin{ytableau}
2
\end{ytableau}.
\end{align*}
If the columns $T_1, \ldots, T_r$ ($T_i$ is the $i$th column of $T$) of a tableau $T \in {\rm SSYT}(k, [m])$ are all fundamental tableaux, then $T_1 \le T_2 \le \cdots \le T_r$ in the above described total order.

\begin{lemma} \label{lem:factorization of tableau as product of fundamental tableaux}
For $k, m \in \mathbb{Z}$, every $T \in {\rm SSYT}(k,[m], \sim)$ can be uniquely factorized as a $\cup$-product of fundamental tableaux and there is a unique $T' \in {\rm SSYT}(k,[m], \sim)$ such that $T' \sim T$ and the columns of $T'$ are fundamental tableaux.
\end{lemma}

\begin{proof}
First we prove the existence. It suffices to prove the existence in the case that $T$ is a one-column tableau. Denote by $i_1<\ldots <i_r$ the entries of $T$. If $i_1=1$, then $T \sim T'$, where $T'$ is the union of the fundamental tableaux $T^{(j, i_j)}$, where the entries of $T^{(j, i_j)}$ are $\{1, 2, \ldots, j-1, i_j\}$, $j \in [2,r]$. If $i_1>1$, then $T \sim T'$, where $T'$ is the union of the fundamental tableaux $T^{(j, i_j)}$, $j \in [r]$.

Now we prove uniqueness. Suppose that $T \sim T'$, $T \sim T''$, and the columns of $T', T''$ are fundamental tableaux. Then $T' \sim T''$. It follows that there are trivial tableaux $A, B$ such that $A \cup T' = B \cup T''$. Since the columns of $A, B$ are trivial tableaux and the columns of $T', T''$ are fundamental tableaux, we have that $A = B$. It follows that $T' = T''$ since ${\rm SSYT}(k, [m], \sim)$ is cancellative by Lemma \ref{lem:SSYT is a monoid}. 
\end{proof}

\begin{example}
In ${\rm SSYT}(5, [6], \sim)$, we have that 
\begin{align*}
\scalemath{0.96}{ \begin{ytableau} 1 & 2 \\ 3 & 4 \\ 5 & 6 \end{ytableau} \sim \begin{ytableau} 1 \\ 3 \\ 5 \end{ytableau} \cup \begin{ytableau} 2 \\ 4 \\ 6 \end{ytableau} \cup \begin{ytableau} 1 \\ 2 \end{ytableau} \cup \begin{ytableau} 1  \end{ytableau} \cup \begin{ytableau} 1 \\ 2 \end{ytableau}  = \begin{ytableau} 1 & 1 & 1 & 1 & 2 \\ 2 & 2 & 3 & 4 \\ 5 & 6 \end{ytableau} }.
\end{align*} 
\end{example}

\subsection{Weights on semi-standard tableaux and on products of flag minors} \label{subsec:weights on tableaux and product of flag minors}

There is a bijection between the set of one-column semi-standard tableaux in ${\rm SSYT(k-1, [k], \sim)}$ and the set of (non-trivial) flag minors of $\CC[N]$ sending the one-column tableau with entries in $J \subset [k]$ to the flag minor $\Delta_J$. Denote by $T_{\Delta}$ the tableau corresponding to a flag minor $\Delta$ and $\Delta_T$ the flag minor corresponding to a one-column tableau $T$. For a tableau $T$ with columns $T_1, \ldots, T_r$, we denote by $\Delta_T = \Delta_{T_1} \cdots \Delta_{T_r}$ the {\sl standard monomial} of $T$. For a fraction $ST^{-1}$ of two tableaux $S,T$, we denote $\Delta_{ST^{-1}} = \Delta_S \Delta_T^{-1}$.

\begin{definition}\label{def:weight of a tableau}
For a fundamental tableau $T^{(i,p)} \in {\rm SSYT(k-1, [k], \sim)}$, $i \in I$, $p \in [i]$, we define the {\sl weight} of the tableau as $\wt(T^{(i,p)}) = \omega_i \in P$, where $\omega_i$ is a fundamental weight of $\mathfrak{g}$. We define $\wt(\mathds{1})=0$. 

For a tableau $T \in {\rm SSYT(k-1, [k], \sim)}$, we define the {\sl weight} of $T$ as $\wt(T) = \sum_{j} \wt(T^{(j)})$, where $T = \cup_{j} T^{(j)}$ is the unique factorization of the tableau $T$ into fundamental tableaux. 
\end{definition}

\begin{definition}\label{def:weight of flag minors}
For a flag minor $\Delta \in \CC[N]$, we define the {\sl weight} of $\Delta$ as $\wt(T_{\Delta})$. For a product $\prod_j \Delta^{(j)}$ of flag minors, we define $\wt(\prod_j \Delta^{(j)}) = \sum_j \wt(\Delta^{(j)})$.
\end{definition}

\subsection{Isomorphism of monoids}

By Theorem \ref{thm:isomorphism between Grothendieck ring and CN}, $\{\Delta_T: T \in {\rm SSYT}(k-1, [k], \sim)\}$ is an additive basis of $\mathbb{C}[N]$, $N \subset SL_k$. Therefore for any module $[L(M)] \in K(\mathcal{C}_{k,\triangle})$,  
\begin{align} \label{eq:expression of Phi(L(M)) in terms of tableaux}
\Phi([L(M)]) = \sum_{T \in {\rm SSYT}(k-1, [k], \sim)} c_T \Delta_T \in \CC[N],
\end{align}
for some $c_T \in \CC^{\times}$. 

Define ${\rm Top}(\Phi([L(M)]))$ to be the tableau which appears on the right hand side of (\ref{eq:expression of Phi(L(M)) in terms of tableaux}) with the highest weight. By the same proof as the proof of Lemma 3.22 in \cite{CDFL} using $q$-character theory, we have that ${\rm Top}(\Phi(L(M)))$ exists for every $L(M) \in K(\mathcal{C}_{k,\triangle})$. Moreover, $\wt(L(M)) = \wt({\rm Top}(\Phi([L(M)])))$. 

We define a map 
\begin{align} \label{eq:tildephi}
\widetilde{\Phi}: \mathcal{P}^+_{k, \triangle} \to {\rm SSYT}(k-1, [k], \sim),  \hspace{.7cm}
M \mapsto {\rm Top}(\Phi(L(M))),
\end{align}
and denote $T_M = \widetilde{\Phi}(M)$.

Recall that for $i \in I$, $p \in [i]$, $T^{(i,p)}$ is the one-column tableau with entries $\{1,2,\ldots, p-1, p+k-i\}$. The following lemma follows from Theorem \ref{thm:isomorphism between Grothendieck ring and CN} and the definition of $\widetilde{\Phi}$.
\begin{lemma}\label{lem: image of fundamental module is one column fundamental tableau, weights of fundamental modules and flag minors}
For fundamental modules $L(Y_{i,i-2p}) \in  \mathcal{C}_{k, \triangle}$, $i \in I$, $p \in [i]$, we have that $\widetilde{\Phi}(Y_{i,i-2p}) = T^{(i,p)}$ and $\wt(Y_{i,i-2p}) = \wt(T^{(i,p)}) = \omega_i$. 
\end{lemma}

Recall that $T_{(l, a)}$ is a one-column fundamental tableau with $l$ rows and whose last entry is $a$, and $T_{(l, a)} = T^{(l+k-a, l)}$. 

By Lemma \ref{lem:factorization of tableau as product of fundamental tableaux}, every $T \in {\rm SSYT}(k-1,[k],\sim)$ has a unique factorization $T\sim \cup_{i=1}^r T_{(l_i, a_i)}$. We define
\begin{align}
\Psi: {\rm SSYT}(k-1, [k], \sim)  \to \mathcal{P}^+_{k, \triangle},  \hspace{.7cm} T \mapsto \prod_{i=1}^r Y_{l_i+k-a_i, k-a_i-l_i}, \label{eq:psi}
\end{align}
and denote $M_T = \Psi(T)$. We will show that $\Psi$ is the inverse of $\widetilde{\Phi}$. 

\begin{theorem} \label{thm: parameterization of simple modules by tableaux}
The map $\widetilde{\Phi} \colon \mathcal{P}^+_{k, \triangle} \to {\rm SSYT}(k-1, [k],\sim)$ is an isomorphism of monoids and its inverse is $\Psi$. 
\end{theorem}

\begin{proof}
We first show that $\widetilde{\Phi}$ is a homomorphism of monoids. By the theory of $q$-characters, for any $M, M' \in \mathcal{P}_{k,\triangle}^+$, we have that
\begin{align} \label{eq:LMLMprime decomposition}
[L(M)] [L(M')] = [L(M M')] + \sum_{\tilde{M}, \wt(\tilde{M}) < \wt(MM')} c_{\tilde{M}} [L(\tilde{M})],
\end{align}
for some $c_{\tilde{M}} \in \ZZ_{\ge 0}$. Since $\Phi: K(\mathcal{C}_{k,\triangle}) \to \CC[N]$ is an algebra isomorphism, we have that 
\begin{align*} 
\Phi(L(M))\Phi(L(M')) = \Phi(L(M M')) + \sum_{\tilde{M}, \wt(\tilde{M}) < \wt(MM')} c_{\tilde{M}} \Phi(L(\tilde{M})).
\end{align*}
It follows that ${\rm Top}(\Phi(L(M))\Phi(L(M'))) = {\rm Top}(\Phi(L(MM')))$. Therefore $\widetilde{\Phi}(MM')=\widetilde{\Phi}(M) \cup \widetilde{\Phi}(M')$. 

We now show that $\Psi$ is a homomorphism of monoids. Since $\Psi(T)$ only depends on the equivalence class of $T$, it suffices to check that $\Psi(T)\Psi(T') = \Psi(T \cup T')$ when $T,T'$ are tableaux whose columns are fundamental tableaux. It is clear that the columns of the product $T \cup T'$ are also fundamental tableaux. By definition, the value of $\Psi$ on a tableau whose columns are fundamental tableaux is product of the values of $\Psi$ on every column of the tableau. It follows that $\Psi(T)\Psi(T') = \Psi(T \cup T')$. 

We now check that both composites $\Psi \widetilde{\Phi}$ and $\widetilde{\Phi} \Psi$ are the identity map. It suffices to check this on generators. For any $i \in I$, $p \in [i]$, by Lemma \ref{lem: image of fundamental module is one column fundamental tableau, weights of fundamental modules and flag minors} and the definition of $\Psi$, we have
\begin{align*}
& \Psi \tilde{\Phi}(Y_{i,i-2p}) = \Psi(T^{(i,p)}) = \Psi(T_{(p, k+p-i)}) = Y_{i,i-2p}.
\end{align*}
Every fundamental tableau in ${\rm SSYT}(k-1, [k], \sim)$ is a one-column tableau of the form $T_{(l, a)}$ for some $a \in [2, k]$ and $l \in [a-1]$. We have 
\begin{align*}
& \tilde{\Phi} \Psi(T_{(l, a)}) = \tilde{\Phi}(Y_{l+k-a, k-a-l}) = T^{(l+k-a, l)} = T_{(l, a)}.
\end{align*}
\end{proof}

In Table \ref{table:correspondence: monomials, tableaux in SSYTkm, tableaux in Rkm, minors in CN}, the first column consists of all fundamental modules in $\mathcal{C}_{5, \triangle}$ and the second column consists of the corresponding fundamental tableaux in ${\rm SSYT}(4, [5], \sim)$.
%and the third column consists of the corresponding flag minors in $\CC[N]$.

\begin{table}
\begin{equation*}
\begin{tabular}{|c|c|c|c|}
\hline
module & tableau \\
\hline
 $L(Y_{1,-1})$ & $\{5\}$    \\
\hline 
 $L(Y_{2,0})$ & $\{4\}$   \\
\hline
 $L(Y_{2,-2})$ & $\{1,5\}$  \\
\hline
 $L(Y_{3,1})$ & $\{3\}$  \\
\hline
 $L(Y_{3,-1})$ & $\{1,4\}$   \\
\hline
 $L(Y_{3,-3})$ & $\{1,2,5\}$ \\
\hline
 $L(Y_{4,2})$ & $\{2\}$  \\
\hline
 $L(Y_{4,0})$ & $\{1,3\}$   \\
\hline
 $L(Y_{4,-2})$ & $\{1,2,4\}$ \\
\hline
 $L(Y_{4,-4})$ & $\{1,2,3,5\}$ \\
\hline
\end{tabular}
\end{equation*}
\caption{Correspondence between fundamental monomials and fundamental tableaux in ${\rm SSYT}(4, [5], \sim)$. Since all tableaux in the table are one-column tableaux, we represent them by their entries.}
\label{table:correspondence: monomials, tableaux in SSYTkm, tableaux in Rkm, minors in CN}
\end{table}

%By Theorems 1.1, 1.2, and 6.1 in \cite{HL15} and Theorem \ref{thm: parameterization of simple modules by tableaux}, we have the following.
%\begin{theorem}
%The dual canonical basis of $\CC[N]$ (respectively, $\widetilde{\mathbb{C}[SL_k]^{N^-}}$) is parameterized by $\SSYT(k-1, [k], \sim)$.
%\end{theorem}
 
\begin{definition}
For a tableau $T \in \SSYT(k-1, [k], \sim)$, we define an element $\ch_{\CC[N]}(T) \in \CC[N]$ (resp., $\ch_{\widetilde{\mathbb{C}[SL_k]^{N^-}}}(T) \in \widetilde{\mathbb{C}[SL_k]^{N^-}}$) to be the $\Phi_{\CC[N]}([L(M_T)])$ (resp., $\Phi_{\widetilde{\mathbb{C}[SL_k]^{N^-}}}(T)$). 
\end{definition}

Usually we write $\ch_{\CC[N]}(T)$ (respectively, $\ch_{\widetilde{\mathbb{C}[SL_k]^{N^-}}}(T)$) as $\ch(T)$ when we know that we are working on $\CC[N]$ (respectively, $\widetilde{\mathbb{C}[SL_k]^{N^-}}$). 

By Theorems 1.1, 1.2, and 6.1 in \cite{HL15} and Theorem \ref{thm: parameterization of simple modules by tableaux}, we have that following. 

\begin{theorem} \label{thm:chT form dual canonical basis}
The set $\{\ch_{\CC[N]}(T): T \in {\rm SSYT}(k-1, [k], \sim)\}$ (respectively, $\{\ch_{\widetilde{\mathbb{C}[SL_k]^{N^-}}}(T): T \in {\rm SSYT}(k-1, [k], \sim)\}$) is the dual canonical basis of $\CC[N]$ (respectively, $\widetilde{\mathbb{C}[SL_k]^{N^-}}$).
\end{theorem}

\section{Formula for elements in the dual canonical basis} \label{sec:formula for elements in the dual canonical basis}

In this section, we give an explicit formula for every element $\ch_{\CC[N]}(T)$ (respectively, $\ch_{\widetilde{\mathbb{C}[SL_k]^{N^-}}}$) in the dual canonical basis of $\CC[N]$ (respectively, $\widetilde{\mathbb{C}[SL_k]^{N^-}}$). 

\subsection{Formula for $\ch(T)$} \label{subs:formula for chT}
Let $T \in {\rm SSYT}(k-1,[k],\sim)$ be a tableau which is $\sim$-equivalent to a tableaux $T'$ whose columns are fundamental tableaux and which has $m$ columns. We have that the columns of $T'$ are $T_{(a_i, b_i)}$, $i=1,\ldots,m$, for some $a_1, \ldots, a_m \in [k-1]$, $b_1, \ldots, b_m \in [k]$. Denote ${\bf p}_T = \{ (a_i, b_i): i \in [m]\}$ (as a multi-set). We define ${\bf i}_T =(i_1, \ldots, i_m)$ and ${\bf j}_T = (j_1, \ldots, j_m)$, where $i_1 \leq \dots \leq i_m$ are $a_1, \ldots, a_m$ written in weakly increasing order and $j_1 \leq \dots \leq j_m$ are the elements $b_1, \ldots, b_m$ written in weakly increasing order. For ${\bf c}=(c_1, \ldots, c_m), {\bf d} = (d_1, \ldots, d_m) \in \ZZ^m$, we denote ${\bf p}_{{\bf c}, {\bf d}} = \{ (c_i, d_i): i \in [m] \}$ (as a multi-set). 

Let $S_m$ be the symmetric group on $[m]$. Denote by $\ell(w)$ the length of $w \in S_m$ and denote by $w_0 \in S_m$ be the longest permutation. For ${\bf i} = (i_1, \ldots, i_m) \in \ZZ^m$, denote by $S_{\bf i}$ the subgroup of $S_m$ consisting of elements $\sigma$ such that $i_{\sigma(j)}=i_j$, $j \in [m]$. It is clear that for ${\bf i}, {\bf j} \in \ZZ^m$, ${\bf p}_{w' \cdot {\bf i}, {\bf j}} = {\bf p}_{w \cdot {\bf i}, {\bf j}}$ if and only if $w' \in S_{{\bf j}}wS_{{\bf i}}$. By \cite[Sections 2.4, 2.5]{Bou}, \cite[Proposition 2.3]{Kob}, and \cite[Proposition 2.7]{BKPST}, there is a unique permutation of maximal length in $S_{{\bf j}}wS_{{\bf i}}$. 

For any $T \in {\rm SSYT}(k-1, [k], \sim)$, there exists $w \in S_m$ such that ${\bf p}_{T} = {\bf p}_{w \cdot {\bf i}_T, {\bf j}_T}$. Define $w_T \in S_{{\bf j}_T}wS_{{\bf i}_T}$ to be the unique permutation with maximal length. Then ${\bf p}_{T} = {\bf p}_{w_T \cdot {\bf i}_T, {\bf j}_T}$. It is clear that $w_T$ is also the unique permutation in $S_m$ of maximal length such that ${\bf p}_{T} = {\bf p}_{w_T \cdot {\bf i}_T, {\bf j}_T}$. 
%since ${\bf p}_{T} = {\bf p}_{w_T \cdot {\bf i}_T, {\bf j}_T}$ implies that $w_T \in S_{{\bf j}_T}w_TS_{{\bf i}_T}$. 

\begin{definition}\label{defn:usemicolonT}
Let $T \in {\rm SSYT}(k-1,[k],\sim)$ and $T \sim T'$, where $T'$ has $m$ columns and all the columns are fundamental tableaux. For $u \in S_m$, we  define $\Delta_{u;T} \in \mathbb{C}[N]$ (respectively, $\widetilde{\CC[SL_k]^{N^-}}$) as follows. If $j_a \in [i_{u(a)}, i_{u(a)}+k]$ for all $a \in [m]$, define the tableau $\alpha(u;T)$ to be the semi-standard tableau whose columns are $T_{(i_{u(a)}, j_a)}$, $a \in [m]$, and define $\Delta_{u; T} = \Delta_{\alpha(u;T)} \in \mathbb{C}[N]$ (respectively,  $\Delta_{u; T} = \Delta_{\alpha(u;T)} \in \widetilde{\CC[SL_k]^{N^-}}$) to be the standard monomial of $\alpha(u;T)$ (cf. Section \ref{subsec:weights on tableaux and product of flag minors}). If $j_a \not\in [i_{u(a)}, i_{u(a)}+k]$ for some $a \in [m]$, then the tableau $\alpha(u;T)$ is {\sl undefined} and $\Delta_{u;T} = 0$. 
\end{definition}

\begin{example} \label{example:compute iT, jT, wT}
Let $T = \begin{ytableau} 1 & 2 \\ 3 & 4 \\ 5 & 6 \end{ytableau} \in {\rm SSYT}(5, [6], \sim)$. Then $T \sim T'$, $T' = \begin{ytableau} 1 & 1 & 1 & 1 & 2 \\ 2 & 2 & 3 & 4 \\ 5 & 6 \end{ytableau}$. We have that ${\bf i}_T = (1,2,2,3,3)$, ${\bf j}_T = (2,3,4,5,6)$, and $w_T = s_2s_4$. For $u =s_2 \in S_5$, $\alpha(u; T)$ is the semi-standard tableau whose columns are $T_{(1,2)}$, $T_{(2,3)}$, $T_{(3,4)}$, $T_{(2,5)}$, $T_{(3,6)}$. We have $\Delta_{u;T} = \Delta_{2}\Delta_{15}\Delta_{13}\Delta_{126}\Delta_{124}$. 
\end{example} 

We have the following theorem.

\begin{theorem} \label{thm: chT for SLkN}
Let $T \in {\rm SSYT}(k-1,[k],\sim)$ and $T \sim T'$ for some tableau $T'$ whose columns are fundamental tableaux and which has $m$ columns. Then 
\begin{align} 
& \ch_{\CC[N]}(T) = \sum_{u \in S_m} (-1)^{\ell(uw_T)} p_{uw_0, w_Tw_0}(1) \Delta_{u; T'} \in \CC[N], \label{eq:Kazhdan-Lusztig character formula tableau formula1} \\
& \ch_{\widetilde{\CC[SL_k]^{N^-}}}(T) = \sum_{u \in S_m} (-1)^{\ell(uw_T)} p_{uw_0, w_Tw_0}(1) \Delta_{u; T'} \in \widetilde{\CC[SL_k]^{N^-}}. \label{eq:Kazhdan-Lusztig character formula tableau formula2} 
\end{align} 
\end{theorem}

\subsection{Proof of Theorem \ref{thm: chT for SLkN}} \label{subs:proof of the formula of chT}

Let $F$ be a non-archimedean local field. Complex, smooth representations of $GL_n(F)$ of finite length are parameterized by multisegments \cite{BernZ77, Zel}. A multisegment is a formal finite sum ${\bf m} = \sum_{i=1}^m \Delta_i$ of segments. A segment $\Delta$ is identified with an interval $[a,b]$, $a, b \in \ZZ$, $a \le b$. 

By quantum Schur-Weyl duality \cite[Section 7.6]{CP96b}, there is a correspondence between multisegments and dominant monomials  
\begin{equation}\label{eq:multisegtomonom}
[a, b] \mapsto Y_{b-a+1, a+b-1},  \hspace{1cm} Y_{i,s} \mapsto [\frac{s-i+2}{2}, \frac{s+i}{2}].
\end{equation}
Denote by $M_{\bf m}$ the monomial corresponding to a multisegment ${\bf m}$ and ${\bf m}_M$ the multisegment corresponding to a monomial $M$. 

We interpret $M_{[a,a-1]}$ as the trivial monomial 
$1 \in \mathcal{P}^+$ and interpret $M_{[a,b]}$ with $b < a-1$ as $0$. For any $m$-tuples $(\mu, \lambda) \in \ZZ^m \times \ZZ^m$, we define a multi-set: 
\begin{align*}
{\rm Fund}_M(\mu, \lambda) = \{M_{[\mu_i,\lambda_i]}: i \in [m] \}.
\end{align*}

For $\lambda = (\lambda_1, \ldots, \lambda_m) \in \ZZ^m$, denote by $S_{\lambda}$ the subgroup of $S_m$ consisting of elements $\sigma$ such that $\lambda_{\sigma(i)} = \lambda_i$. For $\mu=(\mu_1, \ldots, \mu_m)$, $\lambda = (\lambda_1, \ldots, \lambda_m) \in \ZZ^m$, we denote ${\bf m}_{\mu, \lambda} = \sum_{i=1}^m [\mu_i, \lambda_i]$. 

For a multisegment ${\bf m}$ with $m$ terms, there exist unique weakly decreasing tuples $\mu_{\bf m},\lambda_{\bf m} \in \ZZ^m$ and unique permutation of maximal length $w_{\bf m} \in S_m$ such that ${\bf m} = {\bf m}_{w_{\bf m} \cdot \mu_{\bf m}, \lambda_{\bf m}}$ (\cite[Sections 2.4, 2.5]{Bou}, \cite[Proposition 2.3]{Kob}, and \cite[Proposition 2.7]{BKPST}). Note that for any $w, w' \in S_m$ and any $\mu,\lambda \in \ZZ^m$, ${\bf m}_{w' \cdot \mu, \lambda} = {\bf m}_{w \cdot \mu, \lambda}$ if and only if $w' \in S_{\lambda}wS_{\mu}$. The element $w_{\bf m} \in S_m$ is also the unique permutation of maximal length in $S_{\lambda_{\bf m}} w_{\bf m} S_{\mu_{\bf m}}$. We write $\lambda_{\bf m} = \lambda_M$, $\mu_{\bf m} = \mu_M$, $w_{\bf m} = w_M$ for $M = M_{\bf m}$.  

\begin{proof}[{\bf Proof of Theorem \ref{thm: chT for SLkN}}]
We will prove the formula (\ref{eq:Kazhdan-Lusztig character formula tableau formula1}) for $\ch_{\CC[N]}(T)$. The proof of the formula (\ref{eq:Kazhdan-Lusztig character formula tableau formula2}) for $\ch_{\widetilde{\CC[SL_k]^{N^-}}}(T)$ is the same. 

For every finite dimensional $U_q(\widehat{\mathfrak{sl}_k})$-module $L(M)$,
we have that 
\begin{align} \label{eq:Kazhdan-Lusztig character formula qcharacter formula}
\chi_q(L(M)) = \sum_{u \in S_m} (-1)^{\ell(uw_M)} p_{uw_0, w_Mw_0}(1) \prod_{M' \in {\rm Fund}_M(u\mu_M, \lambda_M)} \chi_q(L(M')).
\end{align}
This formula (see Section 5.2 in \cite{CDFL}) is obtained from a result due to Arakawa-Suzuki \cite{AS} (see also Section 10.1 in \cite{LM18}, and \cite{BaCi, Hen}) and from the quantum affine Schur-Weyl duality \cite{CP96b}. In (\ref{eq:Kazhdan-Lusztig character formula qcharacter formula}), we interpret $\chi_q(L(M_{[a,a-1]})) = 1$ and $\chi_q(L(M_{[a,b]})) = 0$ if $b < a-1$. 

By \eqref{eq:multisegtomonom} and Theorem \ref{thm: parameterization of simple modules by tableaux}, there is a correspondence between multisegments and tableaux induced by the following correspondence between segments and fundamental tableaux:
\begin{equation}\label{eq:multisegtotab}
[\mu, \lambda] \mapsto T_{(1-\mu, k-\lambda)},  \quad T_{(l, a)} \mapsto [1-l, k-a],
\end{equation}
where $T_{(1-\mu, k-\lambda)}$ is the one-column tableau with entries $\{1,2,\ldots,-\mu, k-\lambda\}$. Denote by $T_{\bf m}$ the tableau corresponding to the multisegment ${\bf m}$ and denote by ${\bf m}_T$ the multisegment corresponding to the tableau $T$. 

Denote ${\bf i}_T=(i_1, \ldots, i_m)$, ${\bf j}_T = (j_1, \ldots, j_m)$. By \eqref{eq:multisegtotab}, we have that $i_a = 1-\mu_a$, $j_a = k-\lambda_a$ for $a \in [k]$. Therefore $w_T$ defined in Subsection \ref{subs:formula for chT} and $w_{{\bf m}_T}$ defined in this subsection are the same. 

Apply the isomorphism $\Phi_{\CC[N]}$ in Theorem \ref{thm:isomorphism between Grothendieck ring and CN} and the isomorphism $\widetilde{\Phi}$ in Theorem \ref{thm: parameterization of simple modules by tableaux} to the formula (\ref{eq:Kazhdan-Lusztig character formula qcharacter formula}), we obtain the formula (\ref{eq:Kazhdan-Lusztig character formula tableau formula1}).

\end{proof}

\begin{remark}
The difference between the formulas for $\ch_{\CC[N]}(T)$ and $\ch_{\widetilde{\CC[SL_k]^{N^-}}}(T)$ is that the flag minors in (\ref{eq:Kazhdan-Lusztig character formula tableau formula1}) are flag minors in $\CC[N]$ while the flag minors in (\ref{eq:Kazhdan-Lusztig character formula tableau formula2}) are flag minors in $\widetilde{\CC[SL_k]^{N^-}}$. 

For example, in $\widetilde{\CC[SL_4]^{N^-}}$ and $\CC[N]$, we have that $\ch(\begin{ytableau}
1 & 3 \\ 2 \\ 4
\end{ytableau})= \Delta_3 \Delta_{124} - \Delta_4 \Delta_{123}$. On the other hand, in $\CC[N]$, this is equal to $x_{13}x_{34}-x_{14} = \Delta_{13,34}$. 
\end{remark}

We give an example of a computation of $\ch(T)$. 

\begin{example} \label{example:compute ch(T)}
We take $T = \begin{ytableau} 1 & 2 \\ 3 & 4 \\ 5 & 6 \end{ytableau} \in {\rm SSYT}(5, [6], \sim)$ as in Example \ref{example:compute iT, jT, wT}. Then ${\bf i}_T = (1,2,2,3,3)$, ${\bf j}_T = (2,3,4,5,6)$, and $w_T = s_2s_4$. By Theorem \ref{thm: chT for SLkN}, we have that
\begin{align} \label{eq:example of ch(T)}
\begin{split}
\ch(T) = \ &  \Delta_{2} \Delta_{1 4} \Delta_{1 3} \Delta_{1 2 6} \Delta_{1 2 5}  + \Delta_{3} \Delta_{1 5} \Delta_{1 2} \Delta_{1 2 6} \Delta_{1 2 4} + \Delta_{2} \Delta_{1 6} \Delta_{1 5} \Delta_{1 2 4} \Delta_{1 2 3} \\
&  + \Delta_{5} \Delta_{1 4} \Delta_{1 2} \Delta_{1 2 6} \Delta_{1 2 3} + \Delta_{4} \Delta_{1 6} \Delta_{1 2} \Delta_{1 2 5} \Delta_{1 2 3}  - \Delta_{3} \Delta_{1 4} \Delta_{1 2} \Delta_{1 2 6} \Delta_{1 2 5} \\
& - \Delta_{2} \Delta_{1 6} \Delta_{1 4} \Delta_{1 2 5} \Delta_{1 2 3} - \Delta_{2} \Delta_{1 5} \Delta_{1 3} \Delta_{1 2 6} \Delta_{1 2 4}  - \Delta_{5} \Delta_{1 6} \Delta_{1 2} \Delta_{1 2 4} \Delta_{1 2 3} \\
& - \Delta_{4} \Delta_{1 5} \Delta_{1 2} \Delta_{1 2 6} \Delta_{1 2 3}.
\end{split}
\end{align}
\end{example}

Recall that in Section \ref{subsec:weights on tableaux and product of flag minors}, for a fraction $ST^{-1}$ of two tableaux $S,T$, we denote $\Delta_{ST^{-1}} = \Delta_S \Delta_T^{-1}$. For $T \in {\rm SSYT}(k-1,[k])$. we have that $T = T'' \cup T'$, where $T'$ is a tableau whose columns are fundamental tableaux and $T''$ is a fraction of two trivial tableaux. Define $\ch'(T) = \Delta_{T''} \ch_{\widetilde{\CC[SL_k]^{N^-}}}(T')$. 
We have the following conjecture.
\begin{conjecture} \label{conj:dual canonical basis of CSLkN^-}
For every $T \in {\rm SSYT}(k-1,[k])$, $\ch'(T) \in \CC[SL_k]^{N^-}$. Moreover, $\{\ch'(T): T \in {\rm SSYT}(k-1,[k])\}$ is the dual canonical basis of $\CC[SL_k]^{N^-}$. 
\end{conjecture}

We give an example to explain Conjecture \ref{conj:dual canonical basis of CSLkN^-}.

\begin{example} \label{example:compute ch prime of T}
We take $T = \begin{ytableau} 1 & 2 \\ 3 & 4 \\ 5 & 6 \end{ytableau} \in {\rm SSYT}(5, [6])$. Then $T = T'' \cup T'$, where
$
T' = \begin{ytableau} 1 & 1 & 1 & 1 & 2 \\ 2 & 2 & 3 & 4 \\ 5 & 6 \end{ytableau}$, $T'' = \frac{\mathds{1}}{\begin{ytableau} 1 & 1 & 1 \\ 2 & 2 \end{ytableau}}$. We have that
\begin{align*}
\ch'(T) & = \frac{\ch(T')}{\Delta_{1}\Delta_{12}\Delta_{12}} = \Delta_{136} \Delta_{245} - \Delta_{126} \Delta_{345}  \in \CC[SL_6]^{N^-},
\end{align*}
where $\ch(T')$ is equal to (\ref{eq:example of ch(T)}).
\end{example}

%\begin{theorem} \label{thm: chT for CN}
%Let $T \in {\rm SSYT}(k-1,[k],\sim)$ and $T \sim T'$ for some tableau $T'$ whose columns are fundamental tableaux and which has $m$ columns. Then 
%\begin{align}\label{eq:Kazhdan-Lusztig character formula tableau formula}
%\ch(T) = \sum_{u \in S_m} (-1)^{\ell(uw_T)} p_{uw_0, w_Tw_0}(1) \Delta_{u; T'} \in \CC[N]. 
%\end{align} 
%\end{theorem}
%
%Different from Theorem \ref{thm: chT for SLkN}, $\Delta_{u;T'}$ in Theorem \ref{thm: chT for CN} is interpreted as a minor in $\CC[N]$. 

\section{Mutation of tableaux} \label{sec:mutations description}

In this section, we give a mutation rule for the cluster algebra $\CC[N]$ (respectively, $\widetilde{\CC[SL_k]^{N^-}}$) using tableaux. 

A finite dimensional $U_q(\widehat{\mathfrak{g}})$-module is called \textit{prime} if it is not isomorphic to a tensor product of two nontrivial $U_q(\widehat{\mathfrak{g}})$-modules (cf. \cite{CP97}). A simple $U_q(\widehat{\mathfrak{g}})$-module $M$ is {\sl real} if $M \otimes M$ is simple (cf. \cite{Lec}). We say that a tableau $T \in \in {\rm SSYT}(k-1,[k],\sim)$ is real (respectively, prime) if $M_T$ is real (respectively, prime). 

By Theorem \ref{thm:chT form dual canonical basis}, every element in the dual canonical basis of $\CC[N]$ (respectively, $\widetilde{\CC[SL_k]^{N^-}}$) is of the form $\ch(T)$, $T \in \in {\rm SSYT}(k-1,[k],\sim)$. 
In \cite{KKKO18, Qin17}, it is shown that cluster monomials in $\CC[N]$ (respectively, $\widetilde{\CC[SL_k]^{N^-}}$) belong to the dual canonical basis and they correspond to real modules in $\mathcal{C}_{k,\triangle}$. The cluster variables in $\CC[N]$ (respectively, $\widetilde{\CC[SL_k]^{N^-}}$) correspond to real prime modules in $\mathcal{C}_{k,\triangle}$. Therefore cluster monomials (respectively, cluster variables) in $\CC[N]$ (respectively, $\widetilde{\CC[SL_k]^{N^-}}$) are also of the form $\ch(T)$, where $T$ is a real (respectively, real prime) tableau in ${\rm SSYT}(k-1,[k],\sim)$. 

In \cite[Section 4]{CDFL}, it is shown that the mutation rule in Grassmannian cluster algebras can be described using semi-standard Young tableaux of rectangular shape. Similarly, we now show that the mutation rule in $\CC[N]$ (respectively, $\widetilde{\CC[SL_k]^{N^-}}$) can be described using semi-standard Young tableaux.

Starting from the initial seed of $\CC[N]$ (respectively, $\widetilde{\CC[SL_k]^{N^-}}$), each time we perform a mutation at a cluster variable $\ch(T_r)$, we obtain a new cluster variable $\ch(T'_r)$ defined recursively by 
\begin{align*}
\ch(T'_r)\ch(T_r) = \prod_{i \to r} \ch(T_i) + \prod_{r \to i} \ch(T_i),
\end{align*}
where $\ch(T_i)$ the cluster variable at the vertex $i$. On the other hand, by Theorem \ref{thm:isomorphism between Grothendieck ring and CN} and the formula (\ref{eq:LMLMprime decomposition}), we have that
\begin{align} \label{eq:decomposition of ch(T)ch(T')}
\ch(T_r) \ch(T'_r) = \ch(T_r \cup T'_r) + \sum_{T''} c_{T''} \ch(T'')
\end{align}
for some $T'' \in {\rm SSYT}(k-1,[k],\sim)$, $\wt(T'')<\wt(T_r \cup T'_r)$, $c_{T''} \in \ZZ_{\ge 0}$. Therefore one of the two tableaux $\cup_{i \to r} T_i$ or $\cup_{r \to i} T_i$ has strictly greater weight than the other, and moreover the one with higher weight is equal to $T_r \cup T'_r$ in ${\rm SSYT}(k-1,[k],\sim)$. Denote by $\max\{\cup_{i \to r} T_i, \cup_{r \to i} T_i \}$ this higher weight tableau. Then 
\begin{align}
T'_r = T^{-1}_r \max\{\cup_{i \to r} T_i, \cup_{r \to i} T_i \}.\label{eq:howtomutate}
\end{align}

\begin{remark}
There is a partial order called {\sl dominance order} in the set of semi-standard Young tableaux (cf. \cite[Section 5.5]{Br07}). 

Let $\lambda = (\lambda_1,\dots,\lambda_\ell)$, $\mu = (\mu_1,\dots,\mu_\ell)$, with $\lambda_1 \geq \cdots \geq \lambda_\ell \geq 0$, $\mu_1 \ge \cdots \ge \mu_{\ell} \ge 0$, be partitions. Then $\lambda \le_{\rm dom} \mu$ in the {\sl dominance order} if $\sum_{j \leq i}\lambda_j \le \sum_{j \leq i}\mu_j$ for $i=1,\dots,\ell$. 

For a semi-standard tableau $T$ in ${\rm SSYT}(k, [m])$ and $i \in [m]$, denote by $T[i]$ the sub-tableau obtained from $T$ by restriction to the entries in $[i]$. For a tableau $T$, let ${\rm sh}(T)$ denote the shape of $T$. For $T,T' \in {\rm SSYT}(k, [m])$ of the same shape, $T \le_{\rm dom} T'$ in the {\sl dominance order} if for every $i \in [i]$, ${\rm sh}(T[i]) \le_{\rm dom} {\rm sh}({T'}[i])$ in the dominance order on partitions. 

%The {\sl content} of a tableau $T$ is the vector $(\nu_1,\dots,\nu_m) \in \ZZ^m$, where $\nu_i$ is the number of $i$-filled boxes in $T$. 
%For a tableau $T$, let ${\rm sh}(T)$ denote the shape of $T$. For $i \in [m]$, let $T[i]$ denote the restriction of $T \in {\rm SSYT}(k,[m])$ to the entries in $[i]$. For $T,T' \in {\rm SSYT}(k, [m])$ with the same content, $T \geq T'$ in the {\sl dominance order} if ${\rm sh}(T[i]) \geq {\rm sh}(T'[i])$ in the dominance order on partitions, for $i=1,\dots,m$. 
 
%For $T,T' \in {\rm SSYT}(k-1,[k],\sim)$, if there exist $S, S' \in {\rm SSYT}(k-1,[k])$ such that $S \sim T$, $S' \sim T'$ and $S, S'$ have the same content, then we say that $T, T'$ have the same content. By a similar proof as the proof of Proposition 3.28 in \cite{CDFL}, for $T,T' \in {\rm SSYT}(k-1,[k],\sim)$ with the same content, $T \leq T'$ in the dominance order if and only if $M_T \leq M_{T'} \in \mathcal{P}^+$ defined in (\ref{partial order of monomials}). Therefore in (\ref{eq:howtomutate}), when computing $\max\{\cup_{i \to k} T_i, \cup_{k \to i} T_i \}$, one can also use the dominance order on tableaux. 

The {\sl content} of a tableau $T \in {\rm SSYT}(k, [m])$ is the vector $(\nu_1,\dots,\nu_m) \in \ZZ^m$, where $\nu_i$ is the number of $i$-filled boxes in $T$. By a similar proof as the proof of Proposition 3.28 in \cite{CDFL}, for $T,T' \in {\rm SSYT}(k-1,[k])$ with the same content and with the same shape, $T \leq_{\rm dom} T'$ in the dominance order if and only if $M_T \leq M_{T'} \in \mathcal{P}^+$ in the monomial order in (\ref{partial order of monomials}).

In the mutation described above, if we use tableaux in ${\rm SSYT}(k-1, [k])$ (not other tableau representatives of equivalence classes in ${\rm SSYT}(k-1, [k], \sim)$), then in every step, $\cup_{i \to r} T_i$ and $\cup_{r \to i} T_i$ have the same shape and the same content. Therefore in the mutations, one can also use tableaux in ${\rm SSYT}(k-1, [k])$ and use the dominance order on tableaux to compute $\max\{\cup_{i \to r} T_i, \cup_{r \to i} T_i \}$ in (\ref{eq:howtomutate}).
\end{remark}

%\begin{remark}
%We expect that the mutation rule works not only in $\widetilde{\CC[SL_k]^{N^-}}$ but also in $\CC[SL_k]^{N^-}$ by using tableaux in ${\rm SSYT}(k-1,[k])$ and the conjectural formula of $\ch'(T)$ in Conjecture \ref{conj:dual canonical basis of CSLkN^-}. 
%\end{remark}

\begin{example} \label{example:exchange relations}
The following are some examples of exchange relations in $\mathbb{C}[N]$, $N \subset SL_6$, (respectively, $\widetilde{\CC[SL_6]^N}$): $\ch(T_1)\ch(T_2) = \ch(T_3)\ch(T_4)\ch(T_5) + \ch(T_6)\ch(T_7)\ch(T_8)$, where $T_i$'s are the following tableaux respectively
\begin{align*}
& \begin{ytableau}
3  \\
4 \\
5 
\end{ytableau}, \ \begin{ytableau}
2 & 4  \\
3 & 5 \\
4 \\
6
\end{ytableau}, \ \begin{ytableau}
2 \\
3 \\
4 \\
5
\end{ytableau}, \ \begin{ytableau}
4  \\
5 \\
6 
\end{ytableau}, \ \begin{ytableau}
3  \\
4  
\end{ytableau},  \ \begin{ytableau}
3 \\
4 \\
5 \\
6 
\end{ytableau}, \ \begin{ytableau}
2  \\
3 \\
4 
\end{ytableau}, \ \begin{ytableau}
4  \\
5  
\end{ytableau},
\end{align*}
and $\ch(S_1)\ch(S_2) = \ch(S_3)\ch(S_4)\ch(S_5) + \ch(S_6)\ch(S_7)$, where $S_i$'s are the following tableaux respectively
\begin{align*}
\begin{ytableau}
2  \\
4  
\end{ytableau}, \  \begin{ytableau}
1 & 1 & 4  \\
2 & 2 & 5 \\
3 & 5 \\
4 & 6 \\
6
\end{ytableau}, \ \begin{ytableau}
1  \\
2  
\end{ytableau}, \ \begin{ytableau}
2 \\
4 \\
5 \\
6
\end{ytableau}, \ \begin{ytableau}
1 & 4 \\
2 & 5 \\
3 \\
4 \\
6
\end{ytableau}, \ \begin{ytableau}
1 & 4  \\
2 & 5 \\
4 \\
6 
\end{ytableau}, \ \begin{ytableau}
1 & 2  \\
2 & 5 \\
3 \\
4 \\
6 
\end{ytableau}.
\end{align*}
\end{example}

\begin{example}
The cluster variables (not including frozen variables) of $\CC[N]$, $N \subset SL_5$, (respectively, $\widetilde{\CC[SL_5]^N}$) are indexed by the following tableaux:
\begin{align*}
& \begin{ytableau} 2 \end{ytableau}, \: \begin{ytableau} 3 \end{ytableau}, \:\begin{ytableau} 4 \end{ytableau}, \:\begin{ytableau} 1\\ 3 \end{ytableau}, \: \begin{ytableau} 1\\ 4 \end{ytableau}, \:\begin{ytableau} 1\\ 5 \end{ytableau}, \:\begin{ytableau} 2\\ 3 \end{ytableau}, \:\begin{ytableau} 2\\ 4 \end{ytableau}, \: \begin{ytableau} 2\\ 5 \end{ytableau},  \:\begin{ytableau} 3\\ 4 \end{ytableau}, \:\begin{ytableau} 3\\ 5 \end{ytableau}, \: \begin{ytableau} 1\\ 2\\ 4 \end{ytableau}, \:\begin{ytableau} 1\\ 2\\ 5 \end{ytableau}, \:\begin{ytableau} 1\\ 3\\ 4 \end{ytableau}, \: \begin{ytableau} 1\\ 3\\ 5 \end{ytableau}, \: \begin{ytableau} 1\\ 4\\ 5 \end{ytableau}, 
\end{align*}
\begin{align*}
\begin{ytableau} 2\\ 3\\ 4 \end{ytableau},  \:\begin{ytableau} 2\\ 3\\ 5 \end{ytableau}, \: \begin{ytableau} 2\\ 4\\ 5 \end{ytableau},  \: \begin{ytableau} 1\\ 2\\ 3\\ 5 \end{ytableau}, \:\begin{ytableau} 1\\ 2\\ 4\\ 5 \end{ytableau}, \:\begin{ytableau} 1\\ 3\\ 4\\ 5 \end{ytableau}, \ T_1 =\begin{ytableau} 1 & 3\\ 2 \\ 4  \end{ytableau},   \: T_2= \begin{ytableau} 1 & 3\\ 2 \\ 5  \end{ytableau},  \: T_3 = \begin{ytableau} 1 & 4\\ 2 \\ 5  \end{ytableau}, \: T_4 = \begin{ytableau} 1 & 4\\ 3 \\ 5  \end{ytableau}, 
\end{align*}
\begin{align*}
& T_5 = \begin{ytableau} 2 & 4\\ 3 \\ 5  \end{ytableau},    \: T_6 = \begin{ytableau} 1 & 2\\ 3 & 4\\ 5  \end{ytableau}, \: T_7 = \begin{ytableau} 1 & 3\\ 2 & 5\\ 4  \end{ytableau}, \ T_8 = \begin{ytableau} 1 & 3\\ 2 \\ 4 \\ 5  \end{ytableau}, \: T_9 = \begin{ytableau} 1 & 4\\ 2 \\ 3 \\ 5  \end{ytableau}, \: T_{10} = \begin{ytableau} 1 & 1\\ 2 & 4\\ 3 \\ 5  \end{ytableau},
\end{align*}
\begin{align*}
&  T_{11} = \begin{ytableau} 1 & 2\\ 2 & 4\\ 3 \\ 5  \end{ytableau}, \: T_{12} =  \begin{ytableau} 1 & 3\\ 2 & 4\\ 3 \\ 5  \end{ytableau},  \: T_{13} = 
\begin{ytableau} 1 & 3\\ 2 & 4\\ 4 \\ 5  \end{ytableau}, \: T_{14} = \begin{ytableau} 1 & 3\\ 2 & 5\\ 4 \\ 5  \end{ytableau}.
\end{align*}

For a one-column tableau $T$, we have that $\ch(T)$ is the flag minor $\Delta_{i_1, \ldots, i_m}$, where $i_1< \cdots < i_m$ is the set of entries in $T$. For tableaux above with two or more columns, we have that
\begin{align*}
& \ch(T_1) = \Delta_{124}\Delta_3 - \Delta_{123} \Delta_4, \\
& \ch(T_2) = \Delta_{125}\Delta_3 - \Delta_{123} \Delta_5, \\
& \ch(T_3) = \Delta_{125}\Delta_4 - \Delta_{123} \Delta_4, \\
& \ch(T_4) = \Delta_{125}\Delta_{13}\Delta_4 + \Delta_{123}\Delta_{14}\Delta_5 - \Delta_{123}\Delta_{15}\Delta_4 - \Delta_{124}\Delta_{13}\Delta_5, \\
& \ch(T_5) = \Delta_{123}\Delta_{14}\Delta_{2}\Delta_5 + \Delta_{12}\Delta_{124}\Delta_3\Delta_5 + \Delta_{125}\Delta_{13}\Delta_2\Delta_4 \\
& \qquad \qquad - \Delta_{123}\Delta_{15}\Delta_2\Delta_4 - \Delta_{12}\Delta_{125}\Delta_3\Delta_4 - \Delta_{124}\Delta_{13}\Delta_2\Delta_5,
\end{align*}

\begin{align*}
& \ch(T_6) = \Delta_{1 2} \Delta_{1 2 4} \Delta_{1 5} \Delta_{3}- \Delta_{1 2} \Delta_{1 2 3} \Delta_{1 5} \Delta_{4}- \Delta_{1 2 4} \Delta_{1 3} \Delta_{1 5} \Delta_{2} \\
& \qquad \qquad - \Delta_{1 2} \Delta_{1 2 5} \Delta_{1 4} \Delta_{3}+ \Delta_{1 2} \Delta_{1 2 3} \Delta_{1 4} \Delta_{5}+ \Delta_{1 2 5} \Delta_{1 3} \Delta_{1 4} \Delta_{2}, \\
& \ch(T_7) = \Delta_{1 2 3} \Delta_{1 4} \Delta_{5}- \Delta_{1 2 3} \Delta_{1 5} \Delta_{4}- \Delta_{1 2 4} \Delta_{1 3} \Delta_{5}+ \Delta_{1 2 4} \Delta_{1 5} \Delta_{3}, \\
& \ch(T_8) =  \Delta_{1 2 3 5} \Delta_{1 2 4} \Delta_{3}+ \Delta_{1 2 3} \Delta_{1 2 3 4} \Delta_{5}- \Delta_{1 2 3 4} \Delta_{1 2 5} \Delta_{3}- \Delta_{1 2 3} \Delta_{1 2 3 5} \Delta_{4}, \\
& \ch(T_9) = \Delta_{1 2 3 5} \Delta_{4} - \Delta_{1 2 3 4} \Delta_{5},
\end{align*}

\begin{align*}
& \ch(T_{10}) =   \Delta_{1 2 3 5} \Delta_{1 4} - \Delta_{1 2 3 4} \Delta_{1 5}, \\
& \ch(T_{11}) =  \Delta_{1 2 3 5} \Delta_{1 4} \Delta_{2}+ \Delta_{1 2} \Delta_{1 2 3 4} \Delta_{5} 
- \Delta_{1 2 3 4} \Delta_{1 5} \Delta_{2}- \Delta_{1 2} \Delta_{1 2 3 5} \Delta_{4}, \\
& \ch(T_{12}) = \Delta_{1 2 3 5} \Delta_{1 4} \Delta_{3}+ \Delta_{1 2 3 4} \Delta_{1 3} \Delta_{5} 
- \Delta_{1 2 3 4} \Delta_{1 5} \Delta_{3}- \Delta_{1 2 3 5} \Delta_{1 3} \Delta_{4}, \\
& \ch(T_{13}) = \Delta_{1 2 3 4} \Delta_{1 2 5} \Delta_{1 3} \Delta_{4} - \Delta_{1 2 3} \Delta_{1 2 3 4} \Delta_{1 5} \Delta_{4} 
- \Delta_{1 2 3 4} \Delta_{1 2 5} \Delta_{1 4} \Delta_{3} \\
& \qquad \qquad + \Delta_{1 2 3} \Delta_{1 2 3 4} \Delta_{1 4} \Delta_{5}- \Delta_{1 2 3 5} \Delta_{1 2 4} \Delta_{1 3} \Delta_{4}+ \Delta_{1 2 3 5} \Delta_{1 2 4} \Delta_{1 4} \Delta_{3}, \\
& \ch(T_{14}) =   \Delta_{1 2 3 5} \Delta_{1 2 4} \Delta_{1 5} \Delta_{3}+ \Delta_{1 2 3 4} \Delta_{1 2 5} \Delta_{1 3} \Delta_{5}
+ \Delta_{1 2 3} \Delta_{1 2 3 5} \Delta_{1 4} \Delta_{5}  \\
& \qquad \qquad - \Delta_{1 2 3} \Delta_{1 2 3 5} \Delta_{1 5} \Delta_{4}- \Delta_{1 2 3 4} \Delta_{1 2 5} \Delta_{1 5} \Delta_{3}- \Delta_{1 2 3 5} \Delta_{1 2 4} \Delta_{1 3} \Delta_{5}. 
\end{align*}

\end{example}

\section{Application to classification of cluster variables in $\mathbb{C}[N]$} \label{sec:application to classification of cluster variables}

In this section, we apply the results in previous sections to classify cluster variables in $\mathbb{C}[N]$, in the case of $N \subset SL_6$, up to $4$-column tableaux.

We say that a tableau is of rank $r$ if the tableau has $r$ columns. We say that a tableau $T$ is a cluster variable if $\ch(T)$ is a cluster variable. We say that a tableau $T$ is real if $\ch(T)$ satisfies $\ch(T \cup T) = \ch(T)\ch(T)$. That is, $T$ is real if and only if the corresponding $U_q(\widehat{\mathfrak{sl}}_k)$-module $L(M_T)$ is real. 

For $r \in \mathbb{Z}_{\ge 2}$, we call $T_1, \ldots, T_r \in \SSYT(k,[n])$ compatible if $\ch(T_1) \cdots \ch(T_r)=\ch(T_1 \cup \cdots \cup T_r)$. 

By \cite{KKKO18} and \cite{Qin17}, all cluster variables are real and prime. Therefore we first classify real prime tableaux in $\SSYT(k-1, [k])$. 

All rank $1$ tableaux are cluster variables. There are $52$ rank $1$ cluster variables (not including frozen variables) in $\mathbb{C}[N]$, $N \subset SL_6$. These tableaux are in $\SSYT(5, [6])$. 

There are $1652$ semistandard tableaux of rank $2$ in $\SSYT(6, [6])$. There are $1533$ compatible pairs of rank $1$ tableaux in $\SSYT(6, [6])$.
Therefore there are $119$ prime tableaux of rank $2$ in  $\SSYT(6, [6])$. These tableaux are all in $\SSYT(5, [6])$. Among them $118$ tableaux are cluster variables. The only rank $2$ prime tableau which is not cluster variable is $T = \scalemath{0.7}{ \begin{ytableau} 1 & 3 \\ 2 & 5 \\ 4 \\ 6 \end{ytableau} }$. We have that
\begin{align} \label{eq:chT of non-real element in CN}
\begin{split}
\ch(T) = & -2 \Delta_{1 2 3}  \Delta_{1 2 3 4}  \Delta_{1 6}  \Delta_{5}  - \Delta_{1 2 3 5}  \Delta_{1 2 4}  \Delta_{1 6}  \Delta_{3}  - \Delta_{1 2 3}  \Delta_{1 2 3 6}  \Delta_{1 5}  \Delta_{4}- \Delta_{1 2 3 4}  \Delta_{1 2 5}  \Delta_{1 3} \Delta_{6} \\
&  - \Delta_{1 2 3 4}  \Delta_{1 2 6}  \Delta_{1 5}  \Delta_{3} - \Delta_{1 2 3}  \Delta_{1 2 3 5}  \Delta_{1 4}  \Delta_{6}- \Delta_{1 2 3 6}  \Delta_{1 2 4}  \Delta_{1 3}  \Delta_{5} + 2 \Delta_{1 2 3}  \Delta_{1 2 3 4}  \Delta_{1 5}  \Delta_{6} \\
& + \Delta_{1 2 3 6}  \Delta_{1 2 4}  \Delta_{1 5}  \Delta_{3} +\Delta_{1 2 3}  \Delta_{1 2 3 5}  \Delta_{1 6}  \Delta_{4} + \Delta_{1 2 3 4}  \Delta_{1 2 6}  \Delta_{1 3}  \Delta_{5} \\
& + \Delta_{1 2 3 4}  \Delta_{1 2 5}  \Delta_{1 6}  \Delta_{3} + \Delta_{1 2 3}  \Delta_{1 2 3 6}  \Delta_{1 4}  \Delta_{5} + \Delta_{1 2 3 5}  \Delta_{1 2 4}  \Delta_{1 3}  \Delta_{6}.
\end{split}
\end{align}
and
\begin{align*}
\ch(T)\ch(T) = \ch(T \cup T) + \ch( \scalemath{0.7}{ \begin{ytableau} 3 & 5 \\ 4 & 6 \\ 5 \\ 6 \end{ytableau} }).
\end{align*} 

The tableau $T$ corresponds to the simple $U_q(\widehat{\mathfrak{sl}_6})$-module $L(Y_{3, -1} Y_{4, -4} Y_{4, 2} Y_{5, -1} )$, see Theorem \ref{thm: parameterization of simple modules by tableaux}. 

A non-real element in the canonical basis of $U_q(\mathfrak{n})$, $\mathfrak{n}$ is a maximal nilpotent subalgebra of $\mathfrak{sl}_6$, is given in Section 2.7 in \cite{Lec}. Using the generalized quantum affine Schur-Weyl duality \cite{CP96b, KKKOSelecta}, and compare the initial quivers in Figure \ref{fig:initial cluster C[N] k is 5}, Figure \ref{fig:initial cluster labeled by modules}, and the figure in Theorem 1.2 in \cite{Cas20}, we have the following correspondence among fundamental l-weights, good Lyndon words, positive roots in Table \ref{table:correspondence fundamental monomial, positive roots, good Lyndon}, where in the table $\alpha_{ij} = \alpha_i+\cdots +\alpha_j$, $\alpha_i$'s are simple roots. Therefore the vector $b$ in Section 2.7 in \cite{Lec} corresponds to the non-real $U_q(\widehat{\mathfrak{sl}_6})$-module $L(Y_{2,-2}Y_{3,-5}Y_{3,1}Y_{4,-2})$. This module is very similar to $L(Y_{3, -1} Y_{4, -4} Y_{4, 2} Y_{5, -1} )$ that is obtained from the tableau above. 

\begin{table} 
\scalebox{0.88}{
\begin{tabular}{|c|c|c|c|c|c|c|c|c|c|c|c|c|c|c|} 
\hline
$Y_{5,3}$ & $Y_{5,1}$ & $Y_{5,-1}$ & $Y_{5,-3}$ & $Y_{5,-5}$ & $Y_{4,2}$ & $Y_{4,0}$ & $Y_{4,-2}$ & $Y_{4,-4}$ & $Y_{3,1}$ & $Y_{3,-1}$ & $Y_{3,-3}$ & $Y_{2,0}$ & $Y_{2,-2}$ & $Y_{1,-1}$  \\
\hline
$\alpha_1$ & $\alpha_2$ & $\alpha_3$ & $\alpha_4$ & $\alpha_5$ & $\alpha_{12}$ & $\alpha_{23}$ & $\alpha_{34}$ & $\alpha_{45}$ & $\alpha_{13}$ & $\alpha_{24}$ & $\alpha_{35}$ & $\alpha_{14}$ & $\alpha_{25}$ & $\alpha_{15}$ \\
\hline
$(1)$ & $(2)$ & $(3)$ & $(4)$ & $(5)$ & $(12)$ & $(23)$ & $(34)$ & $(45)$ & $(123)$ & $(234)$ & $(345)$ & $(1234)$ & $(2345)$ & $(12345)$ \\
\hline 
\end{tabular} }
\caption{Correspondence among fundamental monomials, positive roots, and good Lyndon words.}
\label{table:correspondence fundamental monomial, positive roots, good Lyndon}
\end{table} 

There are $25740$ semistandard tableaux of rank $3$ in $\SSYT(6, [6])$. There are $21657$ compatible triples of rank $1$ tableaux in $\SSYT(6, [6])$. There are $3913$ compatible pairs of a rank $2$ prime tableau and a rank $1$ tableau in $\SSYT(6, [6])$. 
Therefore there are $170$ prime tableaux of rank $3$ in  $\SSYT(6, [6])$. These tableaux are all in $\SSYT(5, [6])$. All of these tableaux are cluster variables in $\mathbb{C}[N]$. 
  
There are $279279$ semistandard tableaux of rank $4$ in $\SSYT(6, [6])$. There are $212127$ compatible $4$-tuples of rank $1$ tableaux in $\SSYT(6, [6])$. There are $60966$ compatible triples of two rank $1$ tableaux and a rank $2$ prime tableau in $\SSYT(6, [6])$. There are $4322$ compatible pairs of a rank $1$ tableau and a rank $3$ prime tableau in $\SSYT(6, [6])$. There are $1649$ compatible pairs of a rank $2$ prime tableau and a rank $2$ prime tableau in $\SSYT(6, [6])$.  
Therefore there are $215$ prime tableaux of rank $4$ in  $\SSYT(6, [6])$. These tableaux are all in $\SSYT(5, [6])$. Among the $215$ prime tableaux, $214$ of them are cluster variables in $\mathbb{C}[N]$. We conjecture that the remaining one tableau
\begin{align*}
\scalemath{0.7}{
\begin{ytableau} 
1 & 1 & 2 & 4 \\
2 & 3 & 4  \\
3 & 5 & 6  \\
5    \\
6  
\end{ytableau} }
\end{align*}
is not real. 

The computations in this section use SageMath 9.6 \cite{Sage}. The codes and data are available on the webpage: https://sites.google.com/view/jianrong-li. 
 
%Using computer, we checked that the numbers of rank $1,2,3,4$ tableaux (not including frozen variables) which are cluster variables in $\CC[N]$, $N \subset SL_6$, are $52, 118, 170, 212$ respectively. 

\end{document}